\documentclass[11pt]{article}

\usepackage{amsmath,amsthm,amsfonts,amssymb,bm,bbm,enumerate, graphicx, mathtools,mathrsfs,color}
\usepackage{tcolorbox}
\usepackage[english]{babel}
\usepackage[utf8]{inputenc}
\usepackage{fancyhdr,longtable}
\usepackage{caption}
\usepackage{subcaption}
\usepackage[toc,page]{appendix}
\usepackage{csquotes,multirow}

\definecolor{NBrown}{HTML}{66220C}
\definecolor{NAqua}{HTML}{00698C}
\definecolor{ForestGreen}{HTML}{228b22}

\usepackage[backend=bibtex,style=ext-alphabetic,giveninits=true,maxnames=99,maxalphanames=99,autopunct=false]{biblatex}
\usepackage{hyperref}
\hypersetup{colorlinks=true, citecolor=ForestGreen, breaklinks, urlcolor=black, linkcolor=black}
\usepackage[capitalise]{cleveref}
\AtEveryCite{\color{ForestGreen}}

\parskip \medskipamount
\usepackage{tikz}

\usepackage{graphicx}
\usepackage{caption,subcaption}
\newtheorem{theorem}{Theorem}[section]
\newtheorem{lemma}[theorem]{Lemma}
\newtheorem{proposition}[theorem]{Proposition}
\newtheorem{corollary}[theorem]{Corollary}
\newtheorem{remark}[theorem]{Remark}

\newtheorem{definition}[theorem]{Definition}

\newtheorem{conjecture}[theorem]{Conjecture}

\makeatletter
\renewcommand\@dotsep{10000}
\makeatother
\usepackage{mleftright}
\mleftright

\DeclareMathOperator{\diam}{Diam}

\renewcommand{\epsilon}{\varepsilon}

\newcommand{\pr}[1]{\mathbb{P}\!\left(#1\right)}
\newcommand{\E}[1]{\mathbb{E}\!\left[#1\right]}

\newcommand{\Pb}{\mathbb{P}}

\newcommand{\Pcal}{\mathcal{P}}
\newcommand{\indic}[1]{\mathbbm{1}_{\left\{#1\right\}}}

\newcommand{\R}{\mathbb{R}}

\newcommand{\ZZ}{\mathbb{Z}}

\newcommand{\N}{\mathbb{N}}

\newcommand{\I}{\mathcal{I}}
\newcommand{\T}{\mathcal{T}}

\newcommand{\W}{\mathcal{W}}

\newcommand{\J}{\hat{J}}

\newcommand{\rhoo}{\partial}
\newcommand{\Volt}{\upsilon}

\newcommand{\zetas}{\zeta^{[s]}}

\newcommand{\Wts}{\widehat{W}^{[s]}}

\newcommand{\Ws}{{W}^{[s]}}
\newcommand{\Wt}{\widehat{W}}
\newcommand{\Noo}{\N^{(1)}}
\newcommand{\No}{N^{(1)}}
\newcommand{\Ns}{N^{(s)}}
\def\X{X^{\text{exc}}}
\newcommand{\tf}{\mathfrak{t}}

\newcommand{\q}{\mathfrak{q}}
\def\Ta{\T_{\alpha}}
\newcommand{\Tn}{T_n^{(\alpha)}}
\newcommand{\Mn}{M_n^{(\alpha)}}

\newcommand{\dmn}{d_n}
\newcommand{\dt}{d_{\alpha}}
\def\Ma{\mathbf{m}_{\alpha}}
\newcommand{\pof}{\mu_{\alpha}}
\newcommand{\Da}{D_{\alpha}}
\newcommand{\M}{M}
\newcommand{\rootm}{\rho}

\newcommand{\Ito}{It\^o }

\def\Levy{L\'evy }
\def\BGW{Bienaym\'e--Galton--Watson }
\def\cadlag{c\`{a}dl\`{a}g }
\def\diam{\textsf{Diam}}
\newcommand{\thub}{t_{\text{hub}}}

\allowdisplaybreaks

\begin{document}

\title{Stable quadrangulations and stable spheres}
\author{{Eleanor Archer}\footnote{Modal'X, UMR CNRS 9023, UPL, Univ. Paris-Nanterre, F92000 Nanterre, France.\\ \phantom{aaa} \href{mailto:earcher@parisnanterre.fr}{\texttt{earcher@parisnanterre.fr}}, \href{mailto:laurent.menard@normalesup.org}{\texttt{laurent.menard@normalesup.org}}}, {Ariane Carrance}\footnote{CMAP, UMR CNRS 7641, \'Ecole polytechnique, F9112 Palaiseau, France.\\ \phantom{aaaa}\href{mailto:ariane.carrance@math.cnrs.fr}{\texttt{ariane.carrance@math.cnrs.fr}}},\, and {Laurent Ménard}\footnotemark[1]}

\date{\today}

\maketitle

%%%----------------------------------------------------------------------------------------------------------

\begin{abstract}
We consider scaling limits of random quadrangulations obtained by applying the Cori--Vauquelin--Schaeffer bijection to \BGW trees with stably-decaying offspring tails with an exponent $\alpha \in (1,2)$. We show that these quadrangulations admit subsequential scaling limits which all have Hausdorff dimension $\frac{2\alpha}{\alpha-1}$ almost surely. We conjecture that the limits are unique and spherical, and we introduce a candidate for the limit that we call the $\alpha$-\textit{stable sphere}. In addition, we conduct a detailed study of volume fluctuations around typical points in the limiting maps, and show that the fluctuations share similar characteristics with those of stable trees.

% \bigskip

% \noindent{\bf MSC 2010 Classification:}
% 05A15, %    Exact enumeration problems, generating functions
% 05A16, %    Asymptotic enumeration
% 05C12, %    Distance in graphs
% 05C30, %    Enumeration in graph theory
% 60C05, %    Combinatorial probability
% 60D05, %    Geometric probability and stochastic geometry
% 60K35, %    Interacting random processes; statistical mechanics type models; percolation theory
% 82B44  %    Disordered systems (random Ising models, random Schrödinger operators, etc.)
\end{abstract}

\begin{figure}[h!]
\centering
 \subfloat{\quad\quad
      \includegraphics[width=0.27\textwidth,page=1]{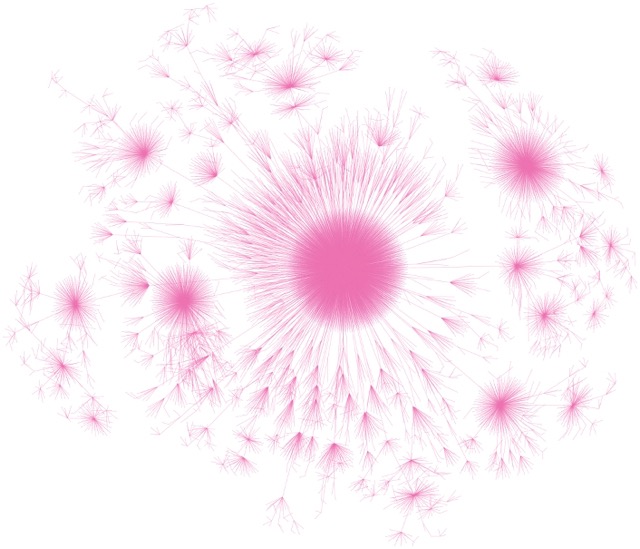}
      \label{subfig:treesim1}\quad\quad} 
    \subfloat{\quad\quad
      \includegraphics[width=0.35\textwidth,page=2]{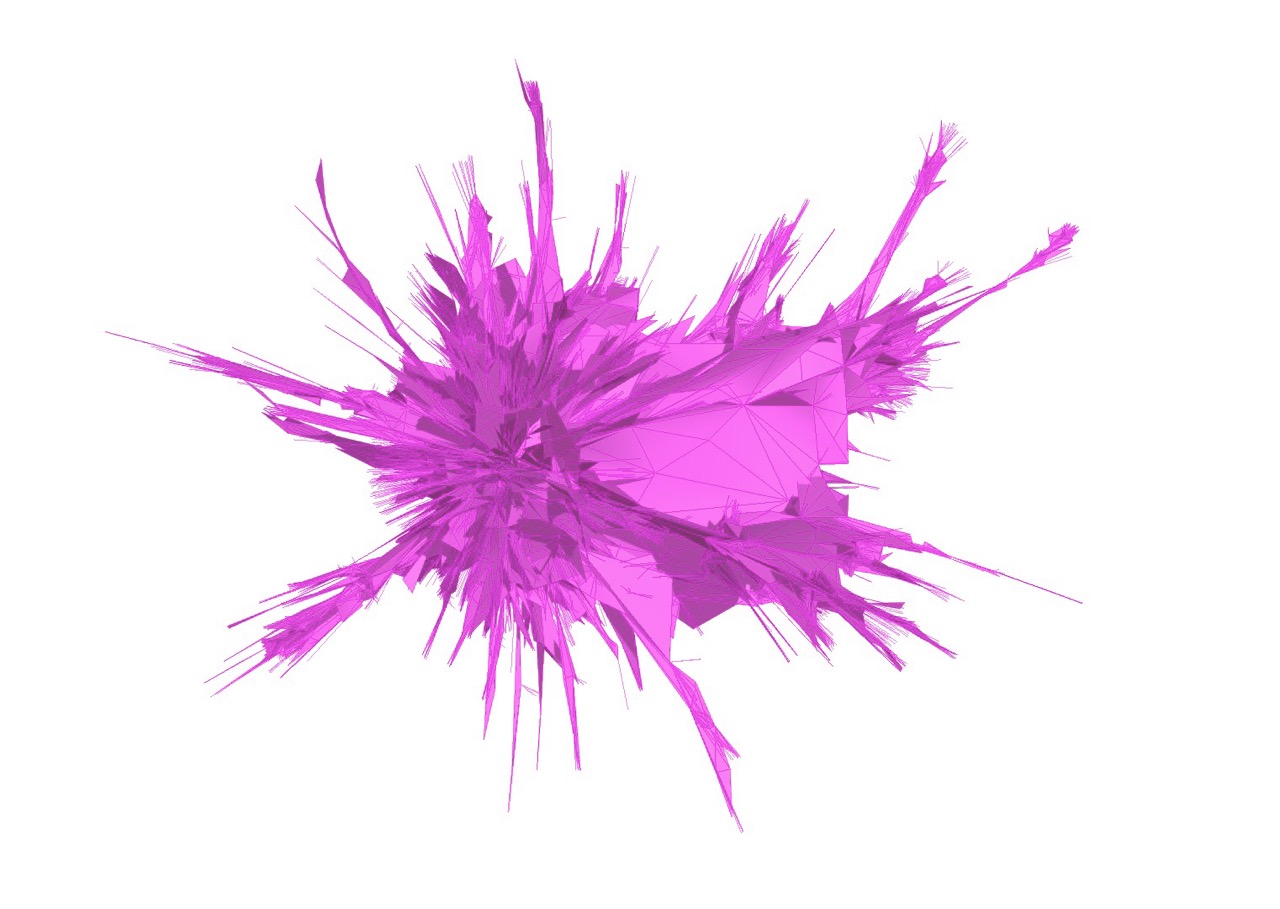}
      \label{subfig:quadsim1}\quad\quad}
    \caption{A stable tree with index $\alpha = 1.2$ and 12699  edges and its associated quadrangulation.}
  \subfloat{\quad\quad
      \includegraphics[width=0.27\textwidth,page=1]{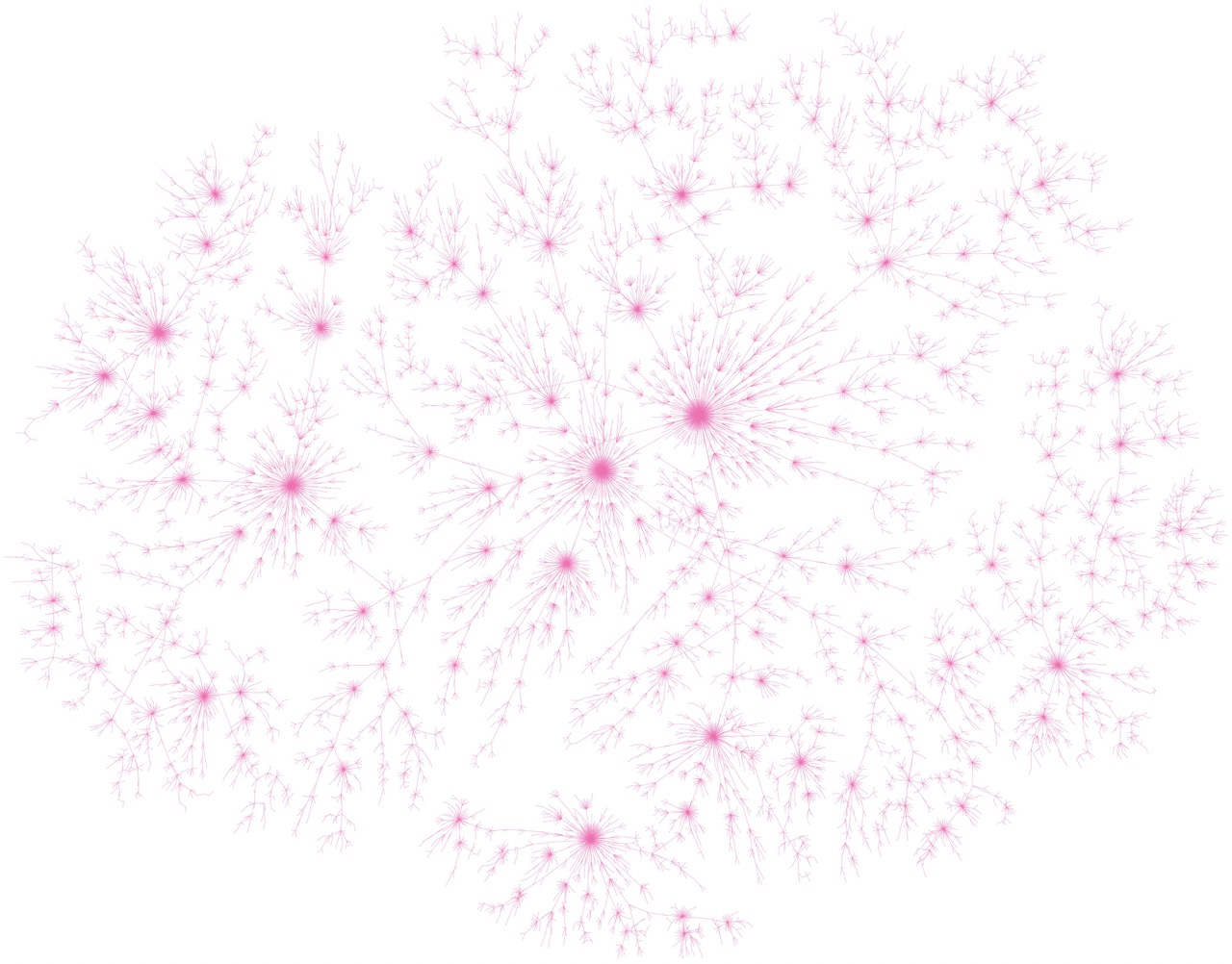}
      \label{subfig:treesim2}\quad\quad} 
    \subfloat{\quad\quad
      \includegraphics[width=0.35\textwidth,page=2]{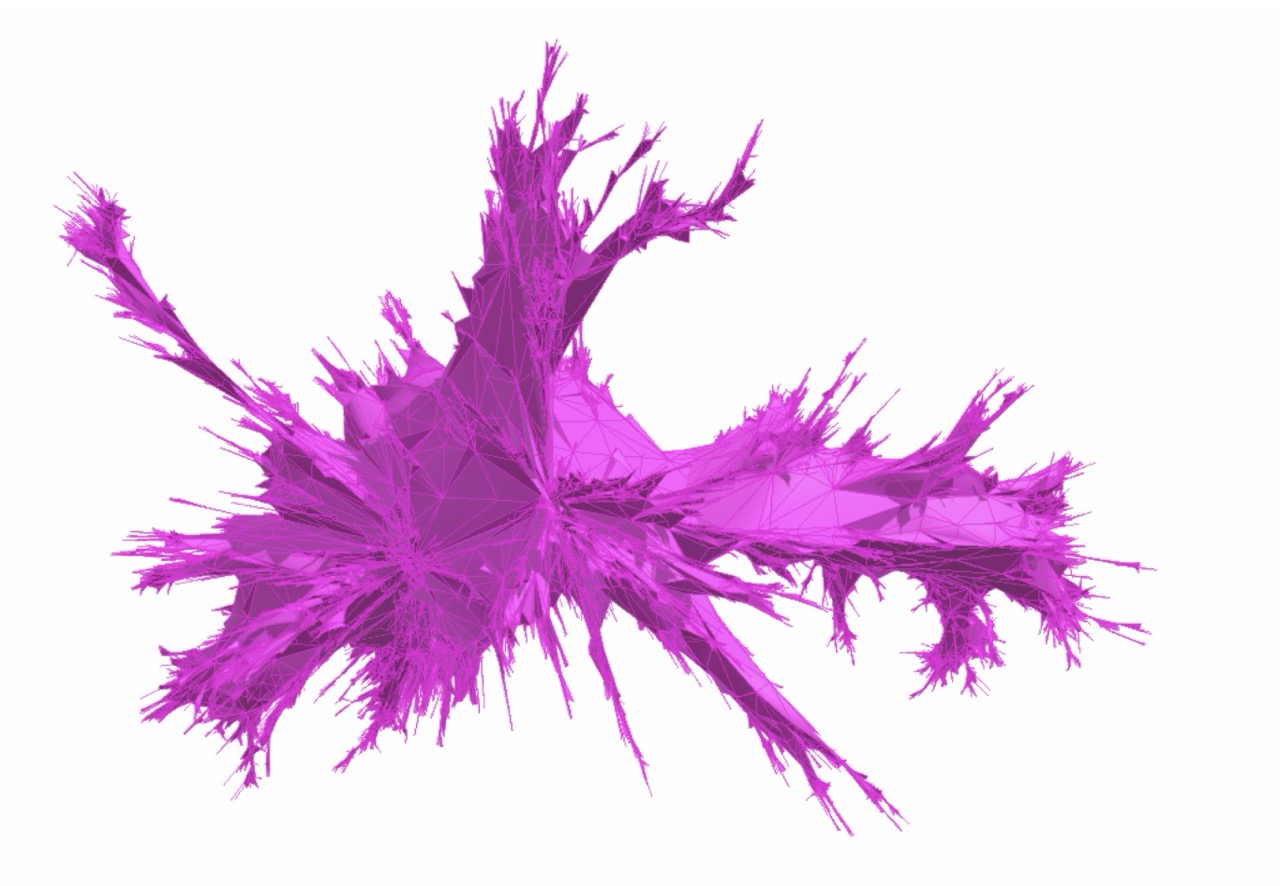}
      \label{subfig:quadsim2}\quad\quad}
          \caption{A stable tree with index $\alpha = 1.6$ and 22185  edges and its associated quadrangulation. }
    \label{fig:Simulation}
\end{figure}

\tableofcontents

%\newpage

\section{Introduction}

The purpose of this paper is to introduce a stable model of random planar quadrangulations and to start to investigate their scaling limits. The starting point is the celebrated convergence of uniform planar quadrangulations to the Brownian sphere  \cite{le2013uniqueness, miermont2013brownian}, which utilises the well-known Cori--Vauquelin--Schaeffer (CVS) bijection between labelled trees and planar quadrangulations. In particular, a quadrangulation can be represented as the image of a labelled tree under the CVS bijection. Each edge of the tree becomes a face of the quadrangulation and the labels determine distances in the latter object. In this paper, we consider sequences of random quadrangulations obtained by applying the CVS bijection to labelled Bienaym\'e--Galton--Watson trees for which the offspring distribution has stably-decaying tails.

The Brownian sphere was first constructed in \cite{marckert2006limit} as a candidate for the scaling limit of uniform quadrangulations with $n$ faces, along with various associated convergence results. This followed on from the work of Chassaing and Schaeffer who already showed that the rescaled radius of a uniform quadrangulation converges to the width of the support of the Brownian snake \cite{chassaing2004random}. The breakthrough in terms of its metric structure came in the landmark works of Le Gall \cite{LeGall2007BrownianMapTopological,legall-scalinglim-geod,le2013uniqueness} and Miermont \cite{miermont2013brownian} who showed that the Brownian sphere is the unique scaling limit, in the Gromov-Hausdorff sense, of uniform $2p$-angulations, and triangulations, with $n$ faces. Moreover, it is almost surely homeomorphic to the sphere \cite{LeGall2007BrownianMapTopological,le2008scaling,miermont2008sphericity}, and its Hausdorff dimension is equal to 4 almost surely \cite{LeGall2007BrownianMapTopological}. Since these pioneering results, the universality class of random models of planar maps whose scaling limit is the Brownian sphere has been extended to general maps~\cite{bettinelli-jacob-miermont}, bipartite maps~\cite{abraham-bipartite}, simple triangulations and quadrangulations~\cite{aba-simple}, Boltzmann maps and maps with prescribed degree sequences in the Brownian regime~\cite{marzouk-prescribed-degrees-2022}, all odd $p$-angulations~\cite{aba-odd}, bicolored triangulations~\cite{carrance-trig}, cubic planar graphs~\cite{stufler, afl-cubic} and block-weighted quadrangulations~\cite{fleurat-salvy}.

The crucial ingredient in the proofs of both Le Gall and Miermont were bijections between labelled trees and random planar maps obtained earlier by Cori-Vauqielin \cite{cori1981planar}, Schaeffer \cite{schaeffer1998conjugaison} and generalised by Bouttier, Di Francesco and Guitter \cite{BDG}. In particular, a uniform quadrangulation with $n$ faces can be sampled by taking a uniform rooted random plane tree with $n$ edges, sampling a so-called label function that evolves as a simple random walk along each branch of the tree, and applying the CVS bijection to this rooted labelled tree. It should therefore be of no surprise that the Brownian sphere can be constructed from scaling limits of such labelled trees. The scaling limit of such a label process is known as the Brownian snake, which is informally Brownian motion indexed by the Brownian continuum random tree (CRT) \cite{LeGallBook}. The Brownian sphere can be constructed as a quotient space of the CRT, and the precise form of the quotient is determined by the associated Brownian snake. Moreover, the CRT plays the role of the cut-locus for the Brownian sphere. 

The fact that the label process converges to the Brownian snake is therefore completely key to the scaling limit results of Le Gall and Miermont. Thanks to recent results of Marzouk \cite{marzouk-brownianstable}, we also now understand how to take scaling limits of labelled trees in a much more general setting. In particular, he considers labelled critical \BGW trees with offspring distribution in the domain of attraction of an $\alpha$-stable law for some $\alpha \in (1,2)$. The metric space scaling limits of such trees (without labels) are well-known to be the \textit{stable trees} introduced by Le Gall and Le Jan \cite{le1998branching} and developed by Duquesne and Le Gall \cite{LeGDuqMono} -- see Figure~\ref{fig:Simulation} for simulations. Like the CRT, these are continuous random trees but unlike the CRT they have the property that they contain branch points of infinite degree, known as hubs. In \cite{marzouk-brownianstable} Marzouk shows that these discrete trees can be additionally endowed with a label process that converges to ``Brownian motion indexed by the stable tree". This will be referred to as the \textit{stable snake} throughout this article.

This therefore begs the question of whether it is possible to emulate the construction of the Brownian sphere using a stable snake and to obtain convergence of the associated discrete model of random quadrangulations to this limiting object. The present paper is a natural step towards that aim. We consider convergence in the pointed Gromov-Hausdorff-Prokhorov topology (this is a straightforward extension of the Gromov-Hausdorff topology in which we also consider convergence of measures and roots). Just as in the Brownian case, the convergence of snakes of Marzouk ensures that we have tightness in this topology. The space of rooted compact finite metric measure spaces is Polish \cite{athreya2016gap}, hence this implies that these maps have subsequential limits, but we do not a priori know that limits are unique. However we conjecture that, similarly to the Brownian case, the limits are indeed unique and we moreover have a candidate for the limit, which is constructed from the stable snake analogously to the construction of the Brownian sphere from the Brownian snake (we plan to address this in future papers). We call this limit the \textit{stable sphere}. In the present paper, we show that any subsequential limit of the discrete quadrangulations has Hausdorff dimension equal to $\frac{2\alpha}{\alpha-1}$ almost surely, and obtain precise results on the corresponding volume fluctuations.

Let us now briefly describe our model. For $\alpha \in (1,2)$, we write $\Ta$ for the stable tree with index $\alpha$ as introduced by Le Gall and Le Jan \cite{le1998branching} and developed by Duquesne and Le Gall \cite{LeGDuqMono}; see Section~\ref{section:stable-trees} for details. The case $\alpha = 2$ corresponds to the CRT, introduced by Aldous \cite{AldousCRTI, AldousCRTII, AldousCRTIII}. We will denote by $\rho$ the root of $\Ta$ and by $\dt$ the distance on the tree $\Ta$. It will be useful to have in mind that $\Ta$ has both a genealogical and a lexicographical order, which is inherited from the left to right ordering of its branches in the plane.

The stable snake of Duquesne and Le Gall \cite{LeGDuqMono} allows us to assign Brownian labels to the vertices of $\Ta$ as follows : given $\Ta$, we consider a centred Gaussian process $(Z_v)_{v \in \Ta}$ such that $Z_\rho = 0$ and such that $\mathrm{Var} (Z_v - Z_{v'}) = \dt (v,v')$ for every $v,v' \in \T_\alpha$. We will review this process in detail in Section~\ref{section:stablesnake}. Similarly to the Brownian sphere, we can define a pseudo-distance $D^*_\alpha$ on $\T_\alpha$ from these labels. First, we define the following mapping
\begin{align}
\Da^\circ \, : \,  \Ta \times \Ta & \to  \mathbb R_+ \nonumber\\
 (v,v') & \mapsto  \Da^\circ (v,v') = Z_v + Z_{v'} - 2 \inf_{u \in I_{v,v'}} Z_u,
\label{eq:def-Do}
\end{align}
where $I_{v,v'}$ denotes the lexicographical interval between $v$ and $v'$ that minimises $\Da^\circ (v,v')$. Indeed, for any pair of vertices $v,v' \in \Ta$, there are always two lexicographical intervals between $v$ and $v'$ corresponding to the two possible ways to go from $v$ to $v'$ around the tree.

The function $\Da^{\circ}$ does not satisfy the triangle inequality, which leads us to introduce
\begin{equation}
\Da^* (v,v')  :=\inf\left\{\sum_{i=1}^k \Da^{\circ} (v_{i-1},v_i) \, \Big\lvert \, k\geq 1, \, v_i \in \Ta \, \forall i\in \{0,\dots,k\}, v_0=v, v_k=v'\right\},
\label{eq:def-D*}
\end{equation}
which is by construction the largest pseudo-distance that satisfies $\Da^* \leq \Da^{\circ}$.
We also let $\approx$ denote the equivalence relation on $\Ta$ defined by $v \approx v'$ if and only if $\Da^*(v,v')=0$ and denote by $\Ma$ the quotient space $\Ta / \approx$. We can also endow it with a measure $\nu_{\alpha}$ and a root $\rho_{\alpha}$, defined as the pushforwards of the canonical measure and root of $\Ta$ under the quotient operation. The \textbf{stable sphere} is the random compact rooted metric measure space $\left(\Ma ,\Da^*, \nu_{\alpha}, \rho_{\alpha}\right)$. Note that the special case $\alpha = 2$ is the Brownian sphere, and that we will focus on $\alpha <2$. Like in the Brownian setting, we expect the stable tree $\Ta$ to play the role of the cut-locus of the stable sphere.

Our discrete model of random quadrangulations can be defined as follows. Fix $\alpha \in (1,2)$ and let $\Tn$ be a critical Bienaymé-Galton-Watson tree with offspring distribution $\pof$ in the domain of attraction of an $\alpha$-stable law conditioned to have $n$ edges. For ease of notation we will in fact assume that there exists a constant $c \in (0, \infty)$ such that $\pof ([x, \infty)) \sim cx^{-\alpha}$ as $x \to \infty$ (we anticipate however that all of our results can be adapted to incorporate slowly-varying corrections). We add labels to the vertices of $\Tn$ by labelling the root $0$ and choosing i.i.d. centred increments in $\{-1,0,+1\}$ along edges. Denote by $\Mn$ the random rooted quadrangulation with $n$ faces associated to $\Tn$ by the Cori--Vauquelin--Schaeffer correspondence, $\dmn$ the graph distance on $\Mn$, $\nu_n$ the uniform probability measure on its vertices, and $\rho_n$ its root vertex. We conjecture the following convergence:

\begin{conjecture}
For $\alpha \in (1,2)$, there exists a constant $\kappa >0$ that depends on $\pof$, such that the following convergence in law holds in the pointed Gromov-Hausdorff-Prokhorov topology:
\[
\left(\Mn , n^{-\frac{\alpha -1}{2 \alpha}} \, \dmn, \nu_n, \rho_n \right) \underset{n \to \infty }{\rightarrow} \, \left(\Ma , \kappa  \Da^*, \nu_{\alpha}, \rho_{\alpha} \right).
\]
\end{conjecture}

The starting point of this paper is the following result.

\begin{theorem} \label{thm:limits}
The sequence $\left(\Mn , n^{-\frac{\alpha -1}{2 \alpha}} \, \dmn, \nu_n, \rho_n \right)_{n \geq 0}$ is tight in the pointed Gromov-Hausdorff-Prokohorov topology. Moreover, any subsequential limit $(\M,D, \nu, \rho)$ can be defined as a quotient space of $\left(\mathbf{m}_\alpha , D^*_\alpha, \nu_{\alpha}, \rho_{\alpha} \right)$.
\end{theorem}

The main contribution of the paper is in fact a study of volumes of balls at typical points in a limit $(\M,D, \nu, \rho)$. We start with the Hausdorff dimension.

\begin{theorem}\label{thm:Hausdorff dimension}
Any subsequential limit $(\M,D, \nu, \rho)$ has Hausdorff dimension $\frac{2\alpha}{\alpha-1}$ almost surely, as does $\left(\mathbf{m}_\alpha ,D^*_\alpha, \nu_{\alpha}, \rho_{\alpha} \right)$.
\end{theorem}

Our volume results are actually a lot more precise than this. As is the case for stable trees (and for similar reasons), the volumes of balls at typical points experience logarithmic fluctuations above their typical value.

\begin{theorem}\label{thm:limsup fluctuations intro}
Take any $\epsilon>0$. Let $(\M,D,\nu,\rho)$ be a subsequential limit appearing in Theorem \ref{thm:limits}. Then, almost surely, it holds for $\nu$-almost every $x$ in $\M$ that
\[
\limsup_{r \downarrow 0} \frac{\nu (B(x,r))}{r^{\frac{2\alpha}{\alpha-1}}(\log r^{-1})^{\frac{2\alpha+\epsilon}{\alpha-1}}} = 0, \qquad \limsup_{r \downarrow 0} \frac{\nu (B(x,r))}{r^{\frac{2\alpha}{\alpha-1}}(\log r^{-1})^{\frac{1-\epsilon}{\alpha-1}}} = \infty.
\]
\end{theorem}

One can ask whether there is a precise regularly varying gauge function $g(r)$ such that the limit supremum appearing in Theorem \ref{thm:limsup fluctuations intro} is equal to a finite positive constant almost surely. In the case of stable trees, this question has been answered negatively by Duquesne \cite[Proposition 1.9]{Duquesne2010packing} (see also \cite[Theorem 1.10]{Duquesne2010packing} for a similar statement for the exact Hausdorff measure). We have no reason to expect otherwise for $(M, D, \nu, \rho)$.

Since $D \leq \Da^*$, the first result of Theorem \ref{thm:limsup fluctuations intro} clearly also holds for $\left(\mathbf{m}_\alpha ,D^*_\alpha, \nu_{\alpha}, \rho_{\alpha} \right)$, and in fact an inspection of the proof shows that the same is true for the second result.

We now turn to a discussion of fluctuations below typical values. The fluctuations are log-logarithmic, as is again the case for stable trees.

\begin{theorem}\label{thm:liminf fluctuations intro}
Let $(\M,D,\nu,\rho)$ be a subsequential limit appearing in Theorem \ref{thm:limits}. Then, almost surely, for $\nu$-almost every $x$ in $\M$
\[
\liminf_{r \downarrow 0} \frac{\nu (B(x,r))}{r^{\frac{2\alpha}{\alpha-1}}(\log \log r^{-1})^{-\frac{1}{\alpha-1}}} < \infty, \qquad \liminf_{r \downarrow 0} \frac{\nu (B(x,r))}{r^{\frac{2\alpha}{\alpha-1}}(\log \log r^{-1})^{-\frac{32\alpha}{\alpha-1}}} = \infty.
\]
Moreover, almost surely,
\[
\liminf_{r \downarrow 0} \frac{\inf_{x \in M} \nu (B(x,r))}{r^{\frac{2\alpha}{\alpha-1}}(\log r^{-1})^{-\frac{32\alpha}{\alpha-1}}} > 0.
\]
\end{theorem}

It was similarly shown in \cite{duquesne2012exact} that stable trees admit $r^{\frac{\alpha}{\alpha-1}}(\log \log r)^{-\frac{1}{\alpha-1}}$ as a regular packing gauge function, meaning that the limit infimum when renormalising volumes by this function is almost surely a finite, non-zero constant. In addition, Duquesne \cite[Theorem 1.1]{duquesne2012exact} showed that the packing measure associated with the gauge function $r^{\frac{\alpha}{\alpha-1}}(\log \log r^{-1})^{-\frac{1}{\alpha-1}}$ coincides with the natural volume measure on stable trees; we leave it open whether a similar result is true for $(M, D, \nu, \rho)$. The exponent of $\frac{32\alpha}{\alpha-1}$ appearing in the second  and third results is certainly not optimal.

We remark that do not expect global supremal volumes to be as tightly concentrated as global infima; in fact it is fairly easy to show that there will exist balls $B(x,r)$ in $\M$ of volume at least $r^2$. This is due to the fact that there are also polynomial volume fluctuations at hubs in the underlying stable tree (see Remark \ref{rmk:hub volume LB}).

The main tool to obtain lower bounds on volumes is the fact that $D \leq \Da^*$ and hence it is sufficient to lower bound the volumes of balls measured with respect to the metric $\Da^*$. This can be achieved by controlling the fluctuations of the snake in various intervals. In order to obtain upper bounds on volumes of balls, we will instead need a lower bound on $D$. This is done by relating $D$ to the metric on the associated cactus; see Lemma \ref{lem:D snake upper bound} for the precise result.

The heavier upper tails on the volumes allows us to make an important distinction between $(\M, D)$ and the Brownian sphere. Let $v^*$ denote the point in $\Ta$ where the minimum of $Z$ is attained (like in the Brownian case, it is unique almost surely, see Proposition~\ref{prop:minimum unique}). By comparing the bounds in Theorem \ref{thm:limsup fluctuations intro} with bounds on $\nu(B(v^*, r))$ obtained in \cite{archer2024snakes}, we deduce that $v^*$ is not a typical point of $(M, D, \nu, \rho)$.

\begin{corollary}\label{cor:minimum atypical}
The vertex $v^*$ is not a typical point in $(\M, D)$, meaning that $v^*$ cannot be coupled with an independent point chosen according to $\nu$.
\end{corollary}

We end the introduction with some comments on related models. At first sight one might guess that the subsequential scaling limits of $\left(\Mn , n^{-\frac{\alpha -1}{2 \alpha}} \, \dmn, \nu_n, \rho_n \right)_{n \geq 0}$ may be related to the dual maps, or the images by Tutte's bijection, of the stable maps of \cite{LeGallMiermont2010scaling}. Indeed, the latter is a random map model in which the size of typical \emph{faces} have stably-decaying tails, while the typical \emph{vertex} degrees are heavy-tailed in our model (see~\Cref{rem: vertex degrees} for a more detailed discussion). However, our result on the Hausdorff dimension clearly rules this out; while our limits have Hausdorff dimension $\frac{2\alpha}{\alpha-1}$, the stable duals have dimension $\frac{\alpha}{\alpha-3/2}$ when $\frac{3}{2}< \alpha < 2$, and in fact have exponential volume growth when $\alpha < \frac{3}{2}$~\cite{budd-curien-stable-spheres} (in the critical case of $\alpha=\frac{3}{2}$ the volume growth is stretched exponential \cite{kammerer2023large}). Nonetheless, like the limits of Theorem \ref{thm:limits}, scaling limits of these dual graphs are believed to be spherical when $\frac{3}{2}< \alpha < 2$. The stable maps themselves, on the other hand, are not believed to be spherical, and instead have more complicated fractal-like topology.

The tightness of Theorem \ref{thm:limits} is very straightforward to establish; the main contributions of the paper are the volume results. Throughout most of the paper we will assume implicitly that we have already restricted to an appropriate subsequence where the maps converge, without explicitly stating it every time.

\textbf{Organisation.} The paper is organised as follows. In Section \ref{section:stable-trees} we give some background on stable trees and stable snakes, and in Section \ref{section:discrete-setup} we give the definitions of all of the map models. In Section \ref{sctn:uniform points large} and \ref{sctn:uniform points small} we respectively consider the upper and lower fluctuations of volumes around uniform points. Finally we conclude in Section \ref{sctn:Hausdorff dim} with a quick deduction of the Hausdorff dimension.

\textbf{Acknowledgements.} We would like to thank Marie Albenque, Meltem \"Unel, Armand Riera and Alejandro Rosales-Ortiz for helpful conversations, and Nicolas Curien for a comment which led to Corollary \ref{cor:minimum atypical}. EA and LM were supported by the ANR grant ProGraM (ANR-19-CE40-0025).

%\newpage

\section{Background on stable trees and snakes}
\label{section:stable-trees}

\subsection{Stable trees}

\subsubsection{Stable \Levy processes and the \Ito excursion measure}\label{sctn:Levy and Ito}
We consider a spectrally positive stable process \Levy $X$ with index $\alpha \in (1,2)$. Its law started at $0$ is denoted by $P$. This process is characterised by its Laplace transform and the fact that it has no negative jumps. For $\lambda >0$ and $t \ge 0$ we will assume that $X$ is normalised so that
\begin{equation}\label{eqn:Levy process Laplace transform}
\E{ \exp \{- \lambda X_t \}} = \exp \{-t \lambda^\alpha \}.
\end{equation}
(See \cite[Sections VII and VIII]{BertoinLevy} for more background.)

We will write $\psi(\lambda) = \lambda^{\alpha}$ to denote this Laplace exponent. The function $\psi$ is known as the \textbf{branching mechanism}. By the L\'evy--Khinchin formula, $\psi$ can equivalently be written in the form
\begin{equation}\label{eqn:psi Levy Khinchin}
\psi ( \lambda) = \int_0^{\infty} (e^{-\lambda r} -1 + \lambda r) \pi (dr) = C_{\alpha}\lambda + \int_0^{\infty} (e^{-\lambda r} -1 + \lambda r \mathbbm{1}\{y \leq 1\}) \pi (dr),
\end{equation}
where $\pi$ is the \textbf{jump measure} of $X$ and $-C_{\alpha}$ is the \textbf{drift coefficient}. The fact that $\psi(\lambda) = \lambda^{\alpha}$ entails that 
\begin{equation}\label{eqn:pi and C def}
C_{\alpha}=\frac{\alpha-1}{\Gamma (2-\alpha)} \quad \text{ and } \quad \pi (dr) = \alpha C_{\alpha} r^{-\alpha-1} dr.
\end{equation}
In standard terminology, this means that the process $X$ corresponds to the \textbf{\Levy triple} $(0, -C_{\alpha}, \pi)$.

To define the stable tree, we will in fact use a normalised L\'evy excursion rather than a L\'evy process. Let $X$ be a spectrally positive $\alpha$-stable L\'evy process, satisfying \eqref{eqn:Levy process Laplace transform}, and let $I_t = \inf_{s \in [0,t]} X_s$ denote its running infimum process. Define $g_1$ and $d_1$ by
\begin{align*}
g_1 &= \sup \{ s \leq 1: X_s = I_s \} \\
d_1 &= \inf \{ s > 1: X_s = I_s \}.
\end{align*}

Note that $X_{g_1} = X_{d_1}$ almost surely since $X$ almost surely has no jump at time $g_1$ and no negative jumps. We define the normalised excursion $X^{\text{exc}}$ of $X$ above its infimum at time $1$ by
\[
X_s^{\text{exc}} = (d_1 - g_1)^{\frac{-1}{\alpha}} (X_{g_1 + s(d_1 - g_1)} - X_{g_1})
\]
for every $s \in [0,1]$. Then $X^{\text{exc}}$ is almost surely an $\alpha$-stable \cadlag function on $[0,1]$ with $X^{\text{exc}}(s)>0$ for all $s \in (0,1)$, and $X_0^{\text{exc}}=X_1^{\text{exc}}=0$.

We also take this opportunity to introduce the \Ito excursion measure. For full details, see \cite[Chapter IV]{BertoinLevy}, but the measure is defined by applying excursion theory to the process $X - I$, which is strongly Markov and for which the point $0$ is regular for itself. We normalise local time so that $-I$ denotes the local time of $X - I$ at its infimum, and let $(g_j, d_j)_{j \in \mathcal{I}}$ denote the excursion intervals of $X - I$ away from zero. For each $i \in \mathcal{I}$, the process $(e^i)_{0 \leq s \leq d_i-g_i}$ defined by $e^i(s) = X_{g_i + s} - X_{g_i}$ is an element of the excursion space
\[
E = \bigcup_{l > 0} D^{\text{exc}}([0,l], \R^{\geq 0}),
\]
where for every $l >0$, $D^{\text{exc}}([0,l], \R^{\geq 0})$ is the space of all \emph{\cadlag}nonnegative functions $f$ such that $f(0)=f(l)=0$.
We let $\sigma (e) = \sup \{s>0: e(s)>0\}$ denote the \textbf{lifetime} (duration) of the excursion $e$. It was shown in \cite{ItoPP} that the measure on $\mathbb R_+ \times E$
\[
\sum_{i \in \mathcal{I}} \delta (-I_{g_i}, e^i)
\]
is a Poisson point measure of intensity $dt N(de)$, where $N$ is a $\sigma$-finite measure on the set $E$ known as the \textbf{\Ito excursion measure}.

The measure $N(\cdot)$ inherits a scaling property from that of $X$. For any $\lambda > 0$ we define a mapping $\Phi_{\lambda}: E \rightarrow E$ by  $\Phi_{\lambda}(e)(t) = \lambda^{\frac{1}{\alpha}} e(\frac{t}{\lambda})$, so that $N \circ \Phi_{\lambda}^{-1} = \lambda^{\frac{1}{\alpha}} N$ (e.g. see \cite{WataIto}). It follows from the results of \cite[Section IV.4]{BertoinLevy} that we can uniquely define a set of measures $(\Ns, s>0)$ on $E$ such that:
\begin{enumerate}[(i)]
\item For every $s > 0$, $\Ns( \sigma=s)=1$.
\item For every $\lambda > 0$ and every $s>0$, $\Phi_{\lambda}(\Ns) = N^{(\lambda s)}$.
\item For every measurable $A \subset E$,
\begin{equation}\label{eqn:Ito measure integrate s}
N(A) = \int_0^{\infty} \frac{\Ns(A)}{\alpha \Gamma(1 - \frac{1}{\alpha}) s^{\frac{1}{\alpha}+1}} ds.
\end{equation}
\end{enumerate}

$\Ns$ is thus used to denote the law $N( \cdot | \sigma = s)$. The probability distribution $\No$ coincides with the law of $\X$ constructed above. Note that (see \cite[Proposition 5.6]{goldschmidt2010behavior}), for all $s \geq 0$,
\begin{equation}\label{eqn:N lifetime tails}
N( \sigma > s) = c s^{-\frac{1}{\alpha}}.
\end{equation}

\subsubsection{Stable height processes and exploration processes}

We will see later in Section \ref{section:discrete-setup} that discrete plane trees can be coded by different functions. The same is true in the continuum; in the case of $\alpha$-stable trees, we first code using a spectrally positive $\alpha$-stable \Levy process or excursion $X$ (this plays the same role as the \L ukasiewicz path in the discrete setting - we will define this in Section \ref{sctn:discrete tree defs}). The \textbf{height function} $H$ is then defined by setting, for $t \geq 0$,
\begin{equation}\label{eqn:height def}
H_t = \lim_{\epsilon \rightarrow 0} \frac{1}{\epsilon} \int_0^t \mathbbm{1} \{X_s < I_s^t + \epsilon \} ds,
\end{equation}
where $I_{s,t} = \inf_{s \leq r \leq t} X_r$. For each $t \geq 0$, this limit exists in probability by \cite[Lemma 1.2.1]{LeGDuqMono}. Moreover, it follows from \cite[Theorem 1.4.3]{LeGDuqMono} that $H$ almost surely has a continuous modification, and we will assume henceforth that $H$ is indeed continuous.

The height function is not in general Markovian; to overcome this difficulty we introduce the exploration process, a Markovian random measure that encodes the height process. Let $M_f(\R_+)$ denote the set of finite measures on $\R_+$. The \textbf{exploration process} $(\rho_t)_{t\geq 0}$ is a process such that for each $t \geq 0$, $\rho_t$ is the measure in $M_f(\R_+)$ satisfying
\begin{align}\label{def:exploration process}
\langle\rho_t, f\rangle = \int_0^t f(H_s) d_s (I_{s,t})
\end{align}
for all bounded measurable functions $f:\R_+ \to \R$. This implies that
\begin{align}\label{eqn:rhot 1}
\langle\rho_t, 1\rangle = I_{t,t} - I_{0,t} = X_t - I_t.
\end{align}

It follows from the definition that for each $t \geq 0$, $\text{supp } \rho_t = [0,H_t]$. In particular this means that $H$ can be defined as a functional of $\rho$; for a given realisation $\rho'$ we denote the associated height process by $H(\rho')$.

We also introduce the dual exploration process. Note that it follows from \eqref{def:exploration process} that $\rho_t$ can also be defined by
\begin{align*}
\rho_t(dr) = \sum_{0 < s \leq t: X_{s-} \leq I_{s,t}} (I_{s,t}-X_{s-})  \delta_{H_s}(dr)
\end{align*}
for all $t>0, r \in [0,H_t]$. The \textbf{dual exploration process} is similarly defined by
\begin{align}\label{eqn:dual exp process}
{\hat{\rho}}_t(dr) = \sum_{0 < s \leq t: X_{s-} \leq I_{s,t}} (X_{s} - I_{s,t})  \delta_{H_s}(dr).
\end{align}

The height function appearing in \eqref{eqn:height def} in fact codes a whole forest of trees. We will sometimes instead code a single tree by replacing $X$ with $\X$ in \eqref{eqn:height def}. This corresponds to the law of $H$ under the measure $\No$. By scaling, this means that we can also make sense of it under $\Ns$ for any $s>0$, and therefore under $N$ using \eqref{eqn:Ito measure integrate s}. The pair $(\rho_t, {\hat{\rho}}_t)_{t\geq 0}$ can be similarly defined under $\No$ or $N$. Under $N$, it holds that (recall that $\sigma$ denotes the lifetime of the excursion)
\begin{equation} \label{eq:reverserho}
(\rho_t, {\hat{\rho}}_t)_{t\geq 0} \overset{(d)}{=} ({\hat{\rho}}_{(\sigma - t)-}, \rho_{(\sigma - t)-})_{t\geq 0}.
\end{equation}

\subsubsection{Stable trees}\label{sctn:stable tree def}

To define the stable tree under $\No$, we first define $H$ using \eqref{eqn:height def} but with $X$ replaced by $\X$, then for $s<t$ define
\begin{equation}\label{eqn:min st def}
m_{s,t} =  \inf_{s \wedge t \leq r \leq s \vee t} H_r.
\end{equation}
We then define a pseudodistance on $[0,1]$ by
\begin{equation}\label{eqn:distance from height}
\dt(s,t) = H_s + H_t - 2m_{s,t}
\end{equation}
whenever $s \leq t$. We then define an equivalence relation on $[0,1]$ by saying $s \sim t$ if and only if $\dt(s,t) = 0$, and set $\Ta$ to be the quotient space $([0,1]/ \sim, \dt)$. We also define a measure $\lambda$ on $\Ta$ as the image of Lebesgue measure on $[0,1]$ under the quotient operation. 

It follows from the construction that $\Ta$ is an $\R$-tree in the sense of \cite{dress1996t}, essentially meaning that it is a continuum object sharing standard properties of discrete trees (see \cite[Definition 2.1]{DLG05} for the precise definition).

Properties of $\Ta$ are determined by the process $\X$. Letting $p: [0, 1] \rightarrow \Ta$ be the canonical projection, there is a distinguished vertex $\rho = p(0)$ which we call the root of $\Ta$. Also, for $u, v \in \Ta$, we denote by $[[u,v]]$ the unique geodesic between $u$ and $v$, and we say $u \preceq v$ if and only if $u \in [[\rho, v]]$. Given $u, v \in \Ta$, we say $z$ is the \textbf{most recent common ancestor} of $u$ and $v$, written $z = u \wedge v$, if $z$ is the unique element in $\Ta$ with $[[\rho, u]] \cap [[\rho, v]] = [[\rho, z]]$. We say that a subtree is \textbf{grafted to the path} $[[u,v]]$ if it is coded by an excursion interval of $\X$ away from its infimum, and its root (i.e. the projection of the start point of this excursion) is in $[[u,v]]\setminus \{u,v\}$.

The relation $\preceq$ on $\Ta$ can be recovered from $\X$ by defining
\[
s \preceq t \text{ if and only if } s \leq t \text{ and } \X_{s^-} \leq I_s^t.
\]
It is possible to show that the two definitions of $\preceq$ given above define partial orders on $\Ta$ and $[0,1]$ respectively, and are compatible with $p$ in the sense that, if $u, v \in \Ta$, then $u \preceq v$ if and only if there exist $s, t \in [0, 1]$ with $p(s)=u$, $p(t)=v$ and such that $s \preceq t$. Additionally, if $s \wedge t$ is the most recent common ancestor of $s$ and $t$ (with respect to $\preceq$), then it can be verified that $p(s \wedge t) = p(s) \wedge p(t)$.

The \textbf{multiplicity} of a vertex $u \in \Ta$ is defined as the number of connected components of $\Ta \setminus \{u\}$. A vertex $u$ is called a \textbf{leaf} if it has multiplicity one, and a \textbf{branch point} if it has multiplicity at least three. It can be shown that $\Ta$ only has branch points with infinite multiplicity, called \textbf{hubs}, and that $u \in \Ta$ is a hub if and only if there exists $s \in [0, 1]$ such that $p(s)=u$ and $\Delta_s := \X_s - \X_{s^-}>0$. See for instance Theorem 4.6 and Theorem 4.7 of \cite{DLG05}. The quantity $\Delta_s$ gives a measure of the number of children of $u$. 

For any fixed $t \geq 0$, the quantity appearing in \eqref{eqn:rhot 1} can be interpreted as the sum of the sizes of the hubs appearing on the right hand side of the branch from the root to $p(t)$. Similarly $\langle{\hat{\rho}}_t, 1\rangle$ can be interpreted as the sum of the sizes of the hubs appearing on the left hand side of the branch from the root to $p(t)$.

We can similarly define $\Ta$ under $N$, rather than $\No$, by replacing $\X$ with an excursion defined under $N$ in the definition of $H$, or by using \eqref{eqn:Ito measure integrate s}. We can similarly define a forest of trees by retaining $X$ in \eqref{eqn:height def} (each excursion of $X-I$ above its infimum then codes a single tree). It is also a well-known fact that the genealogy of $\Ta$ encodes a continuous-state branching process with branching mechanism $\psi$, but we will not use this connection and refer to \cite{LeGDuqMono} for further background.

\subsection{Stable snake with Brownian spatial displacements}

\subsubsection{Definition} \label{section:stablesnake}
We now give a brief introduction to \Levy snakes in the particular case of snakes on stable trees with Brownian spatial displacements. We will often simply refer to this as a ``stable snake" in what follows. A full introduction in the more general setting of \Levy trees is given in \cite[Section 4]{LeGDuqMono}; we refer there for further background.

Informally, a stable snake is simply a real-valued stochastic process indexed by a stable tree. Just as we had to introduce the exploration process as a Markovian version of the height process, it will be more convenient to keep track of the history of this stochastic process and we therefore instead consider a path-valued process. In particular, we let $\W$ denote the set of continuous real-valued functions defined on a compact interval of the form $[0, \zeta]$ for some $\zeta>0$. Given $w \in \W$, we denote its domain by $[0, \zeta_w]$ and set $\widehat{w} = w(\zeta_w)$. We call $\zeta_w$ the \textbf{lifetime} of $w$ and call $\widehat{w}$ the \textbf{tip} of $w$.

We then define a topology on $\W$ using the distance function
\[
d_{\W}(w, w') = |\zeta_w - \zeta_{w'}| + \sup_{r \geq 0} |w(r \wedge \zeta_{w}) - w'(r \wedge \zeta_{w'})|.
\]
Before defining the stable snake, we first define a snake with Brownian spatial displacements driven by a deterministic continuous function $h:\R_+ \to \R_+$ with initial condition $w_0: [0, h(0)] \to \R_+$. This is a random variable taking values in $\W$, or in other words in the space of path-valued processes (one should have in mind that $h$ plays the role of a height function of a deterministic forest, and the snake is obtained by adding random spatial displacements).

For every $x \in \mathbb R$, we denote by $\Pi_x$ the distribution of a standard linear Brownian motion $\xi = (\xi_t)_{t \geq 0}$ started at $x$. It is shown in \cite[Section 4.1.1]{LeGDuqMono} that for any such choice of $(h, w_0)$ there exists a unique probability measure on $\W$ such that the following three properties hold.  
\begin{enumerate}
\item $\zeta_{W_s} = h(s)$ for all $s \in [0, \sigma]$.
\item \label{item:snake-consistency} For all $s, s' \geq 0$, we have that $W_s(t) = W_{s'}(t)$ for all $t \leq m_{s',s}$.
\item Conditionally on $W_{s'}$, the path $W_s:[0, \zeta_s] \to \R$ is such that $(W_s(m_{s',s}+t) - W_s(m_{s',s}))_{t\in [\zeta_s - m_{s',s}]}$ has the same law as the Brownian motion $(\xi_t)_{t\in [0,\zeta_s - m_{s',s}]}$ under $\Pi_{W_{s'}(m_{s',s})}$.
\end{enumerate}

In other words, conditionally on $W_{s'}$, the path $W_s$ satisfies the second point for $t \leq m_{s, s'}$ (which informally corresponds to the snake along the path from the root to the most recent common ancestor of $p(s)$ and $p(s')$), and for $t > m_{s, s'}$ evolves as an independent linear Brownian motion (which corresponds to the rest of the branch to $p(s)$). Given $h$ and $w_0$ as above we let $Q_{w_0}^{h}$ denote the law of this process.

The \textbf{stable snake with Brownian spatial displacements} is the pair $(\rho, W)$ obtained by replacing $h$ with the random height function $H$ defined from a \Levy process $X$ as in \eqref{eqn:height def}. 

We will also use the notation
\begin{equation} \label{eq:tipDef}
\Wt_s = W_s(\zeta_s)
\end{equation}
to denote the tip of the snake at time $s$. 

We will sometimes refer to this construction as simply the ``the snake" or ``the stable snake" in this article.

\subsubsection{Excursion measures for snakes}

We will need to consider the law of a snake where the exploration process $\rho$ is first sampled under $N$ or $\No$, and with initial condition $W_0(0) = x \in \R$. Informally, under $\No$ this is the law of a stable snake on a single stable tree of total mass $1$, started from the point $x$. Under $N$ this is the law of a stable snake on a stable tree sampled according to the It\^o measure, started from the point $x$. We denote these excursion measures for snakes by $\N_x$ and $\Noo_x$, so that
\begin{align}\label{eqn:Ito snake def}
\N_x (d \rho, d W) = N_0 (d \rho) Q_x^{H(\rho)}(dW).
\end{align}
$\Noo_x$ is defined similarly by replacing $N$ with $\No$ above. It is explained in \cite[Section 2.4]{riera2022structure} that the point $(0,x)$ is regular and recurrent for the process $(\rho, W)$ (which is actually Markov), and moreover that the process $(-I_t)_{t\geq 0}$ is a local time at $0$ for this process, and hence this definition makes sense as the excursion measure of $(\rho, W)$ away from $(0,x)$ associated with the local time $-I$.

By excursion theory, it moreover follows that, if $(\alpha_i, \beta_i)_{i \in \mathcal{I}}$ denote the excursion intervals of $(\rho, W)$ away from $(0,x)$, with corresponding subtrajectories $(\rho^{(i)}, W^{(i)})$, then under $\Pb_x$, the measure
\[
\sum_{i \in \mathcal{I}} \delta (-I_{\alpha_i}, \rho^{(i)}, W^{(i)})
\]
is a Poisson point measure with intensity $\mathbbm{1}_{[0, \infty)}(\ell) d \ell \N_x(d \rho, d W)$ (where $d \ell$ denotes Lebesgue measure).

\bigskip

Note that, under $\mathbb N_x$, the function $(\Wt_s)_{s \in [0, \sigma]}$ can also be viewed as a function of the tree $\Ta$ coded by the height function $H(\rho)$ driving the snake that appears in \eqref{eqn:Ito snake def}. We will follow the convention of \cite{LeGall2007BrownianMapTopological} and denote this function by $(Z_a)_{a \in \Ta}$, so that $Z_a = \Wt_t$ for any $a \in \Ta$ such that $a = p(t)$. We also will sometimes abuse notation and interchangeably use $Z_a$, $Z_t$ and $\Wt_t$.

We denote by $\mathcal{R}$ the \textbf{range} of the stable snake:
\begin{equation} \label{eq:rangeDef}
\mathcal R := \left[ \inf_{a \in \Ta} Z_a, \sup_{a \in \Ta} Z_a \right].
\end{equation}

\bigskip

The measure $\N_x$ also inherits a scaling property from that of $N(\cdot)$ and the Brownian scaling property. For any $\lambda > 0$ we define a mapping $\Phi_{\lambda}: (\rho, W) \to (\rho^{\lambda}, W^{\lambda})$ by 
\begin{align}\label{eqn:Ito scaling}
\rho^{\lambda}_t (dr) = \lambda^{\frac{1}{\alpha}} \rho_{t/\lambda} (d(r/\lambda^{1-\frac{1}{\alpha}})), \qquad W_t^{\lambda}(s) = \lambda^{\frac{1}{2}\left(1-\frac{1}{\alpha}\right)} W_{t/\lambda}(s/\lambda^{1-\frac{1}{\alpha}}),
\end{align}
so that the law of $(\rho^{\lambda}, W^{\lambda})$ under $\N_x$ is equal to $\lambda^{1/\alpha}\N_{\lambda^{\frac{1}{2}\left(1-\frac{1}{\alpha}\right)}x}$.

Note that we can use \eqref{eqn:Ito measure integrate s} and \eqref{eqn:Ito scaling} to calculate that 
\begin{align}\label{eqn:Laplace sigma}
N \left( 1 - e^{-\lambda \sigma} \right) = \lambda^{1/ \alpha}.
\end{align}

We also briefly recall a few useful results. The first is the following from \cite[Equation (38)]{archer2024snakes} (recall the definition of the range from \eqref{eq:rangeDef}):
\begin{equation}\label{eq:rangesnake}
\mathbb N_x \left[ \mathbbm{1}\{0 \in \mathcal R \} \right] = \left( \frac{\alpha + 1}{(\alpha - 1)^2} \, \frac{1}{x^2} \right)^{\frac{1}{\alpha -1}},
\end{equation}
The second concerns the H\"older continuity of the stable snake:

\begin{lemma}\textnormal{(Lemma 6.1 in~\cite{archer2024snakes}).}
\begin{enumerate}[(a)]
\item $\No$ almost surely, for any $\gamma< 1-\frac{1}{\alpha}$, the height function is H\"older continuous with parameter $\gamma$.
\item Take any $b \in (0, 1/2)$. Then $\Noo$ a.e. there exists $\epsilon_0> 0$ such that for every $s,t \geq 0$ with $d(s,t) \leq \epsilon_0$,
\[
|\Wt_t-\Wt_s| \leq d(s,t)^b.
\]
\end{enumerate}
\label{lem:height holder}
\end{lemma}

The third concerns the location of the minimum of the snake over the interval $[0,\sigma]$, which is denoted by $s^*$.

\begin{proposition}[Proposition 2.8 in~\cite{archer2024snakes}]
\label{prop:minimum unique}
$\N$-almost everywhere, the time $s^* \in [0,\sigma]$ at which $\Wt$ attains its minimum is unique.
\end{proposition}

We also take this opportunity to define the \textbf{occupation measure of the snake re-rooted at its minimum} as the random measure $\overline{\mathcal I}$ satisfying
\begin{equation}\label{eqn:ISE def rerooted}
\langle \overline{\mathcal I} , f \rangle = \int_0^{\sigma} f(\Wt_s - \Wt_{s^*}) ds.
\end{equation}

\subsubsection{Uniform re-rooting invariance}

We will use a property of stable snakes know as \textbf{uniform re-rooting invariance}. To define this under $\N_0$ for an excursion with lifetime $\sigma>0$, we follow the presentation of \cite[Section 2.3]{legall-weill} for the Brownian case. We first need to define the notion of a stable tree and of a snake rooted at a time point $s \in [0, \sigma]$. To this end, given $\sigma>0$ we first set
\begin{equation*}
s \oplus r = \begin{cases}
s+r &\text{ if } s+r \leq \sigma, \\
s+r - \sigma &\text{ if } s+r > \sigma.
\end{cases}
\end{equation*}
We then define the re-rooted lifetime and snake-tip processes by
\begin{align}\label{eqn:rerooted height and snake tip}
\zetas_r &= \zeta_s + \zeta_{s \oplus r} - 2\inf_{u \in [s \wedge (s \oplus r), s \vee (s \oplus r)]} \zeta_u \qquad \text{ and } \qquad \Wts_r = \Wt_{s \oplus r} - \Wt_s
\end{align}
for all $r \in [0, \sigma]$. Note that the entire sequence of snake trajectories can be recovered directly from \eqref{eqn:rerooted height and snake tip} via
\begin{align*}
\Ws_r (t) = \Wts_{\sup\{u \leq r: \zetas_u=t\}},
\end{align*}
or equivalently \eqref{eqn:rerooted height and snake tip} does indeed suffice to define the entire re-rooted process.

Informally, $(\zetas_r)_{r \geq 0}$ codes a stable tree rooted at the vertex previously labelled $s$, and $\Wts_r$ codes the spatial displacements along the branches of the new tree from the new root to the vertex which now has label $r$.

The re-rooting invariance statement is as follows.
\begin{proposition}[Proposition 2.1 in~\cite{archer2024snakes}]
\label{prop:uniform rerooting}
For every non-negative function $F$ on $\R_+ \times C(\R_+, \W)$ it holds that
\[
\N_0 \left( \int_0^{\sigma} F(s, \Ws) ds \right) = \N_0 \left( \int_0^{\sigma} F(s, W) ds \right).
\]
\end{proposition}

\subsubsection{Spinal decompositions}\label{sctn:spinal decomp}
We now present an important spinal decomposition formula, which will allow us to express certain integral functionals of the stable snake under $\N_0$ using a decomposition over subtrees and subtrajectories along a branch to a uniform point in the underlying tree.

To this end, we let $(U^{(1)}, U^{(2)})$ be a two dimensional subordinator defined on some probability space $(\Omega_0,\mathcal F_0, P^0)$, started at $(0,0)$ and with Laplace transform (recall that $\psi(\lambda) = \lambda^{\alpha}$)

\begin{align}\label{eqn:U Laplace transform}
E^0[\exp\{-\lambda_1 U^{(1)}_t - \lambda_2 U^{(2)}_t\}] = \begin{cases}
\exp\left\{-\frac{t (\psi(\lambda_1)-\psi(\lambda_2))}{\lambda_1 - \lambda_2}\right\} &\text{ if } \lambda_1 \neq \lambda_2, \\
\exp\{-t \psi'(\lambda_1)\} &\text{ if } \lambda_1 = \lambda_2.
\end{cases}
\end{align}
The marginal of each of $U^{(1)}$ and $U^{(2)}$ is that of a subordinator with Laplace exponent $\tilde{\psi}(\lambda) := \frac{\psi(\lambda)}{\lambda} = \lambda^{\alpha -1}$, and the sum $U^{(1)} + U^{(2)}$ is a subordinator with Laplace exponent $\psi'(\lambda) = \alpha \, \lambda^{\alpha -1}$. Similarly to \eqref{eqn:psi Levy Khinchin}, this entails that
\begin{align}\label{eqn:psi tilde prime Laplace rep}
\begin{split}
\tilde{\psi} ( \lambda) &= \int_0^{\infty} (1-e^{-\lambda r} ) \tilde{\pi} (dr) = \frac{C_{\alpha}\lambda}{2-\alpha} + \int_0^{\infty} (1-e^{-\lambda r} - \lambda r\mathbbm{1}\{r \leq 1\}) \tilde{\pi}(dr), \\
\psi' ( \lambda) &= \int_0^{\infty} (1-e^{-\lambda r} ) \pi' (dr) = \frac{\alpha C_{\alpha}\lambda}{2-\alpha} + \int_0^{\infty} (1-e^{-\lambda r}- \lambda r\mathbbm{1}\{r \leq 1\}) \pi' (dr).
\end{split}
\end{align}
where $\tilde{\pi} (dr) = C_{\alpha} r^{-\alpha} dr$ and $\pi' (dr) = \alpha C_{\alpha} r^{-\alpha} dr$ (these can be calculated directly from \eqref{eqn:psi Levy Khinchin} and \eqref{eqn:pi and C def} but see \cite[Section III.1]{BertoinLevy} for further background on stable subordinators). We note only that $\tilde{\psi}$ can be easily recovered from $\psi'$ by using Fubini's theorem and noting that $U^{(1)}$ is obtained from $U^{(1)} + U^{(2)}$ through the fact that a jump at time $t$ of $U^{(1)}$ corresponds to a uniform portion of the corresponding jump at time $t$ of $U^{(1)} + U^{(2)}$

We will continue to use the notation $\psi, \psi'$ and $\tilde{\psi}$ throughout the paper to refer to these three functions above.

For each $a>0$, let $(J_a, \J_a)$ denote a pair of random measures given by
\begin{align} \label{eq:defsJ}
(J_a, \J_a) (dt) = (\mathbbm{1}_{[0,a]}(t) dU_t^{(1)}, \mathbbm{1}_{[0,a]}(t) dU_t^{(2)})
\end{align}

For $h > 0$, under $E^0 \otimes \Pi_x$, conditionally on the measure $(J_h + \hat J_h)$ and the Brownian trajectory $(\xi_s)_{0\leq s \leq h}$, we consider an independent Poisson point process on $[0,h] \times \mathcal W$, denoted by
\[
\mathcal P := \sum_{i\in I} \delta_{a_i,w_i},
\]
and with intensity
\begin{equation}\label{eqn:PP intensity two sided}
(J_h + \hat J_h)(da) \otimes \mathbb N_{\xi_a} (dw).
\end{equation}
The key result is as follows. In the proposition we let $\mathcal{R}$ denote the \textit{range} of the snake $W$, and we let $\mathbb{E}_{\Pcal}$ denote expectation with respect to the point process $\Pcal$ defined above. Also recall that $\Pi_x$ denotes the law of a standard linear Brownian motion started at $x$.

\begin{proposition}[Proposition 2.4 in~\cite{archer2024snakes}] \label{prop:firstmomentrange ds}
For every $x \in \R$ and every non-negative measurable functional $\phi$ taking values in $M_f(\R_+)^2 \times \W$, we have
\begin{align*}
&\N_x \left[ \int_0^{\sigma} \phi (\rho_s, {\hat{\rho}}_s, W_s,\mathcal R, \sigma) ds \right] \\
&\qquad = \int_0^\infty dh \, E^0 \otimes \Pi_x \left [
\mathbb E_{\mathcal{P}} \left[
\phi \left( J_h, \J_h, (\xi_s)_{0\leq s \leq h}, \overline{\bigcup_{i \in I} \left(\zeta_{a_i} + \mathcal R(w_i) \right)} , \sum_{i \in I} \sigma (w_i)\right)
\middle| (J, \J, \xi)
\right] \right].
\end{align*}
\end{proposition}
Informally, the proposition says that the sizes hubs along the branch from $\rho$ to a uniform point are given by the jumps of the subordinator $(J, \hat{J})$, and that the subtrees grafted to the hubs form a Poisson process coded by the \Ito excursion measure.

In order to carry out calculations for Poisson processes such as those appearing in \cref{prop:firstmomentrange ds}, and with respect to the measures $J$ and $\hat{J}$, we remind the reader of two standard formulas.

Firstly note that, using the Markov property and stationarity of increments, the formula \eqref{eqn:U Laplace transform} easily generalises to the following for continuous nonnegative functions $f$.
\begin{align}\label{eqn:U Laplace transform cts}
\begin{split}
E^0\left[\exp\left\{-\int_0^t f(s) d(U^{(1)} + U^{(2)})_s \right\}\right] &= 
\exp\left\{-\int_0^t \psi'(f(s)) ds \right\} \\
E^0\left[\exp\left\{-\int_0^t f(s) dU^{(1)}_s \right\}\right] &= 
\exp\left\{-\int_0^t \tilde{\psi}(f(s)) ds \right\}.
\end{split}
\end{align}

Secondly, we will use the following formula for Poisson point processes several times.

\begin{proposition}\label{prop:Poisson master formula}
Let $P = \sum_{i \in \mathcal{I}} \delta (\omega_i)$ denote a Poisson point process on a space $\Omega$ with intensity measure $\hat{\mu}$ and let $f$ be a function $\Omega \to \R$. Then
\[
\E{\prod_{i \in \mathcal{I}} f (\omega_i)} = \exp \left\{\int_{\Omega} (f (\omega)-1) \hat{\mu}(d \omega)\right\}.
\]
\end{proposition}

\subsubsection{Local time at level $a$}\label{sctn:exit measures and SMP}
It will also be useful to introduce the notion of the local time at a certain level of $\Ta$. For $a>0$, Duquesne and Le Gall \cite{DLG05} showed that one can construct a local time measure, supported on vertices at height $a$ in $\Ta$, and moreover such that the canonical measure $\Volt$ on $\Ta$ satisfies
\begin{equation}\label{eqn:local a def}
\Volt = \int_0^{\infty} \ell^a da,
\end{equation}
$N$-almost everywhere. A vertex chosen according to $\ell^a$ can therefore be interpreted as a vertex chosen ``uniformly at level $a$'' in $\Ta$.

We will use the following first moment formula, which is a special case of \cite[Proposition 4.3.2]{LeGDuqMono} or \cite[Equation (3.4)]{riera2022structure}.

\begin{proposition} \label{prop:firstmomentdLs}
For every domain $D$, every $x\in D$, and every non-negative measurable functional $\phi$ taking values in $M_f(\R_+)^2 \times \W$, we have
\begin{align*}
\N_x \left[ \int_0^{\sigma} \ell^a (ds) \, \phi (\rho_s, {\hat{\rho}}_s, W_s) \right] =  E^0 \otimes \Pi_x \left( \phi(J_{a}, \J_{a}, (\xi_s)_{0\leq s \leq a}) \right).
\end{align*}
\end{proposition}

Note that this also implies that $N( \langle \ell^a, 1 \rangle)=1$ for all $a>0$.

\section{Trees, maps and scaling limits}
\label{section:discrete-setup}

\subsection{The pointed Gromov-Hausdorff-Prokhorov topology}

To study the asymptotic metric-measure structure of a large random map, we will make use of the \textbf{pointed Gromov-Hausdorff-Prokhorov} distance, defined between (isometry classes of) compact rooted metric-measure spaces. In particular, we let $\mathbb{F}^c$ denote the set of quadruples $(F,R,\mu,\rho)$ such that $(F,R)$ is a compact metric space, $\mu$ is a locally finite Borel measure of full support on $F$, and $\rho$ is a distinguished point of $F$, which we call the root, considered up to root-preserving, measure-preserving isometries.

Let $(E_1,d_1, \nu_1, \rho_1)$ and $(E_2,d_2, \nu_2, \rho_2)$ be two compact rooted metric-measure spaces in $\mathbb{F}^c$. The pointed Gromov-Hausdorff-Prokhorov distance between them is defined as
\begin{equation}\label{eqn:dGHP def}
d_{GHP}(E_1,E_2)=\inf\left\{d_{H}(\varphi_1(E_1),\varphi_2(E_2)) + d(\varphi_1(\rho_1),\varphi_2(\rho_2)) + d_P(\varphi_1^*\nu_1, \varphi_2^*\nu_2) \right\},
\end{equation}
where the infimum is over all isometric embeddings $\varphi_i:E_i\to E$ of $E_1$ and $E_2$ into the same metric space $(E,d)$, $d_H$ is the Hausdorff distance between compact subsets of $E$, and $d_P$ denotes the Prokhorov distance between measures on $E$, as defined in \cite[Chapter 1]{BillsleyConv}. It is well-known (for example, see \cite[Theorem 2.3]{AbDelHoschNoteGromov}) that this defines a metric on the space of equivalence classes of $\mathbb{F}^c$, where we say that two spaces $(F,R,\mu,\rho)$ and $(F',R',\mu',\rho')$ are equivalent if there is a measure and root-preserving isometry between them.

We can also consider the \textbf{pointed Gromov-Hausdorff} distance $d_{GH}(\cdot, \cdot)$ between two spaces in $\mathbb{F}^c$; this is defined by removing the term involving the Prokhorov distance from \eqref{eqn:dGHP def} above. A useful characterisation of the pointed Gromov-Hausdorff distance can be made via \emph{correspondences}.

\begin{definition}
If $(E_1, d_1)$ and $(E_2,d_2)$ are two compact metric spaces, a \textbf{correspondence} between $E_1$ and $E_2$ is a subset $R \subseteq E_1 \times E_2$ such that, for every $x_1 \in E_1$, there exists $x_2 \in E_2$ such that $(x_1,x_2) \in R$, and for every $x_2 \in E_2$, there exists $x_1 \in E_1$ such that $(x_1,x_2) \in R$. 

If $R$ is a correspondence between $(E_1,d_1)$ and $(E_2,d_2)$, the \textbf{distortion} of $R$ is given by
\[
\textnormal{dis}_R:=\sup\{\lvert d_1(x_1,y_1) - d_2(x_2,y_2) \rvert; \, (x_1,x_2), (y_1,y_2) \in R\}.
\]
\end{definition}

As was done in the Brownian case, we will use the following theorem to establish tightness in the Gromov-Hausdorff topology.

\begin{theorem}[Theorem 7.3.25 in~\cite{burago-burago-ivanov}]\label{lem:distortion} Let $(E_1,d_1,\nu_1, \rho_1)$ and $(E_2,d_2,\nu_2, \rho_2)$ be  in $\mathbb{F}^c$. Then
\[
d_{GH}((E_1,d_1,\nu_1, \rho_1),(E_2,d_2,\nu_2, \rho_2))=\frac{1}{2}\inf \textnormal{dis}_R,
\]
where the infimum is taken over the set of correspondences between $E_1$ and $E_2$ containing the pair $(\rho_1, \rho_2)$.
\end{theorem}

In this article, we will prove (subsequential) pointed Gromov-Hausdorff-Prokhorov convergence by first proving  pointed Gromov-Hausdorff convergence using correspondences, and then show Prokhorov convergence of the measures on the appropriate metric space.

\subsection{Discrete trees}\label{sctn:discrete tree defs}

A \textbf{tree} is a connected graph (i.e. with vertices and edges) with no cycles. A \textbf{plane tree} is a tree that is embedded in the sphere (it can be equivalently embedded in the plane). In this paper we will deal mainly with rooted plane trees, meaning that there is a distinguished oriented edge known as the \textbf{root edge}. The initial vertex of the root edge is called the \textbf{root vertex}. This induces the notion of genealogy: we say that $u$ is the \textbf{parent} of $v$, or $v$ is a \textbf{child} of $u$ if $u$ and $v$ are neighbours and $u$ is strictly closer to the root vertex.

In what follows, we denote by $\mathbf{T}$ the set of rooted plane trees, and by $\mathbf{T}_n$ the set of rooted plane trees with $n$ edges.

To describe rooted plane trees, we will make use of the \textbf{Ulam-Harris encoding} (see \cite{neveu1986arbres} for further background), in which vertices of a tree $\tf \in \mathbf{T}$ are labelled by elements of $\bigcup_{n \geq 0}\N^n$, where $\N = \{1, 2, \ldots \}$ and where we set $\N^0=\varnothing$. The root of $\tf$ receives Ulam-Harris label $\varnothing$, and the other labels are determined recursively, with the children of vertex $v=v_1v_2\dotsm v_h \in \N^h$ receiving Ulam-Harris labels $(vi, 1\leq i \leq c_{\tf}(v))$ (where $c_{\tf}(v)$ is the number of children of $v$ in $\tf$), in the left-to-right ordering given by the embedding of $\tf$ in the plane, in which the outer endpoint of the root edge is designated to be the leftmost child of the root. For $v=v_1v_2\dotsm v_h$, we denote by $\lvert v \rvert=h$ the generation of $v$. 

Let $\tf \in \mathbf{T}$, and $u,v \in V(\tf)$. We denote by $[[u,v]]$ the set of vertices of $\tf$ that lie on the unique shortest path from $u$ to $v$ in $\tf$, and $]]u,v[[=[[u,v]]\setminus \{u,v\}$. We also denote by $u \wedge v$ the most recent common ancestor of $u$ and $v$ in $\tf$, which may be defined by the equality $[[\varnothing, u \wedge v ]]=[[\varnothing, u]] \cap [[\varnothing,v]]$. This also leads to the notion of \textbf{lexicographical order}: $u$ comes before $v$ in this ordering, denoted $u \leq v$, if and only if $u_{u \wedge v + 1} \leq v_{u \wedge v + 1}$ in the Ulam-Harris encoding.

\begin{remark}
To avoid confusion, we point out that elsewhere in the literature that we cite, rooted plane trees are often referred to as ``rooted ordered plane trees" where the root is a vertex and the ordering is equivalent to the choice of the leftmost child of the root which, together with the planar embedding, fixes the Ulam-Harris/lexicographical ordering of the vertices. Since we started with a root edge, rather than a root vertex, we have already fixed this ordering. Sometimes they are also just referred to as ``plane trees" and the root vertex and ordering is automatically included in this definition.
\end{remark}

We denote by $\preceq$ the genealogical relation on $\tf$: we write $u \preceq v$ and say $u$ is an \textbf{ancestor} of $v$, or $v$ is a \textbf{descendant} of $u$, if $u \in [[\varnothing, v]]$. We use $\leq$ for the lexicographical order on $\tf$, and $\ll$ for the \emph{reverse} lexicographical order: $u \ll v$ means that the image of $v$ comes before the image of $u$ in the lexicographical order of the mirror tree. More precisely, if neither $u \prec v$ nor $v \prec u$ holds, then $u \ll v$ iff $u \leq v$; if $u \prec v$, then $v \ll u$ (and $u \leq v$). We also use the lexicographical order to define a projection $p_{\tf}$ from $\{0, \ldots, |V(\tf)|-1\}$ to $\tf$: $p_{\tf}(i)$ denotes the $(i+1)^{th}$ vertex in the lexicographical ordering of $\tf$.

A tree $\tf \in \mathbf{T}_n$ can be coded by three different functions as illustrated in Figure~\ref{fig:contourheightfns}. We define the \textbf{contour exploration} of $\tf$ as the sequence $e_0, e_1, \dots, e_{2n-1}$ of oriented edges bounding the unique face of $\tf$, starting with the root edge, and ordered counterclockwise around this face (it therefore traces the ``outer boundary" of the tree in ``left-to-right" fashion). For $0 \leq i \leq 2n-1$, the $i^{th}$ visited vertex in this contour exploration is the initial vertex $e^-_i$ of $e_i$, and we set $C_{\tf}(i)=|e^-_i|$. Note that the contour ordering is such that the ordering of the vertices determined by the times of their first visit by the contour function is the same as their lexicographical ordering.

To define the \textbf{height function}, we now list the $n+1$ vertices of $\tf$ in lexicographical order as $u_0, u_1, \ldots, u_n$. We define the height process $H_{\tf}=(H_{\tf}(j))_{0 \leq j \leq n}$ by
\begin{equation}\label{eqn:height function def disc}
H_{\tf}(j)=\lvert u_j \rvert, \, 0 \leq j \leq n.
\end{equation}
We also define the \textbf{Łukasiewicz path} $(X_{\tf}(m))_{0 \leq m \leq n}$ by setting $X_{\tf}(0) = 0$, then by considering the vertices $u_0, u_1, \ldots, u_n$ in lexicographical order and setting $X_{\tf}(m+1) = X_{\tf}(m) + c_{\tf}(u_m)-1$ (recall that $c_{\tf}(u)$ denotes the number of children of $u$ in $\tf$). The height function and Łukasiewicz path are connected by the following relation.
\begin{equation}\label{eqn:height Luk def}
H_{\tf}(m) = | \{j \in \{0, 1, \ldots, m-1\}:X_{\tf}(j)=\inf_{j \leq k \leq m}X_{\tf}(k) \}|.
\end{equation}
The intuition behind this is that the times $j \in \{0, 1, \ldots, m-1\}$ satisfying $X_{\tf}(j)=\inf_{j \leq k \leq m}X_{\tf}(k)$ correspond exactly to the ancestors of the vertex $u_m$.

Although the height function and the contour function directly encode the geometry of the tree, we will see below that in the setting of random trees, the Łukasiewicz path is often more convenient to work with.

\begin{figure}[h]
\centering
\includegraphics[width=16.5cm, height=5.6cm]{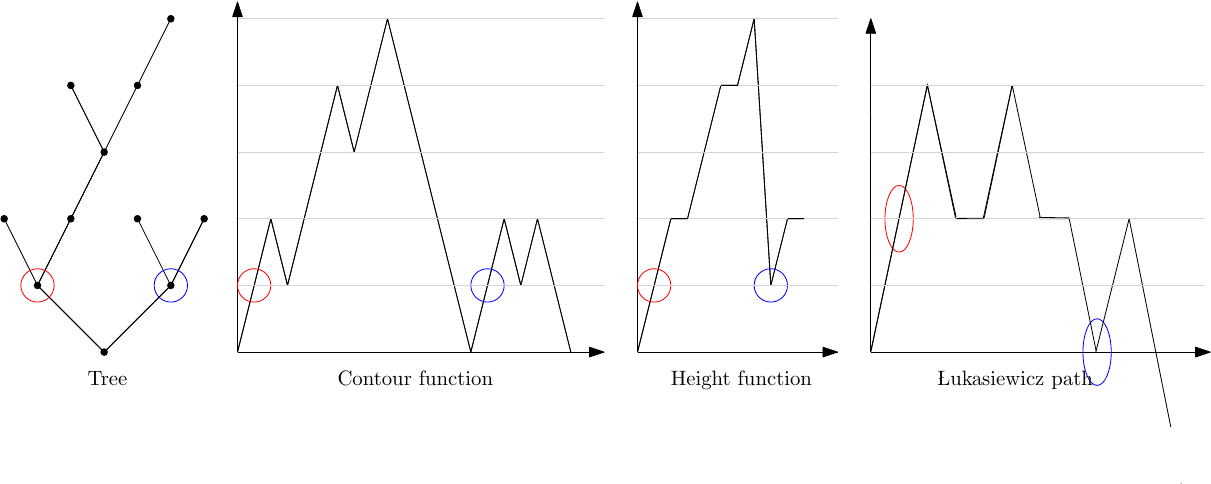}
\caption{Example of contour function, height function and Łukasiewicz path for the given tree.}
\label{fig:contourheightfns} 
\end{figure}

Let $\tf$ be a random element of $\mathbf{T}$. We say it has the \textbf{branching property} if, for any $n \in \N$, conditionally on $\{c_{\tf}(\varnothing)=n\}$, the subtrees $(\tf_1, \dotsm, \tf_n)$ of descendants of the $n$ children of the root are independent and each distributed as the original tree $\tf$. The law of such a tree is characterised by its \textbf{offspring distribution}: for $\pof$ a probability measure on $\ZZ_{\geq 0}$, a random rooted plane tree $\tf$ is a \textbf{Bienaymé--Galton--Watson tree with offspring distribution \boldmath$\pof$} if it has the branching property and the distribution $c_{\tf}(\varnothing)$ of the number of children of the root is $\pof$. We say an offspring distribution is \textbf{critical} if it has mean $1$. Note that Bienaymé--Galton--Watson trees thus come equipped with a root edge and a lexicographical order.

Duquesne \cite[Theorem 3.1]{duquesne2003limit} showed that when $\Tn$ is a Bienaymé--Galton--Watson tree with critical offspring distribution satisfying $\pof[x,\infty) \sim cx^{-\alpha}$ for some $\alpha \in (1,2)$ and $c \in (0, \infty)$, conditioned to have $n$ edges, and $X_n, H_n$ and $C_n$ respectively denote its Łukasiewicz path, height function and contour function, then there exists $c' \in (0, \infty)$ (explicitly depending on $c$) such that the following joint convergence holds with respect to the Skorokhod-$J_1$ product topology:
\begin{equation}\label{eqn:Duquesne contour conv}
\Bigg(\frac{1}{c'}n^{-\frac{1}{\alpha}} X_n(nt), c'n^{-(1-\frac{1}{\alpha})} H_n(nt), c'n^{-(1-\frac{1}{\alpha})} C_n(2nt)\Bigg)_{t \in [0,1]} \xrightarrow[n \to \infty ]{(d)} \Big(\X_t, H_t, H_t \Big)_{t \in [0,1]},
\end{equation}
where $(H_t)_{t \in [0,1]}$ denotes the height function defined from $\X$ as in \eqref{eqn:height def}. (The joint convergence including the Łukasiewicz path is not specifically stated in \cite[Theorem 3.1]{duquesne2003limit}, but is implicit in the proof.) A straightforward consequence is that
\[
(\Tn, c'n^{-(1-\frac{1}{\alpha})}\dmn, \rho_n)  \xrightarrow[n \to \infty ]{(d)} (\Ta, \dt, \rho)
\]
with respect to the pointed Gromov-Hausdorff topology, where $\dmn$ denotes the graph metric on $\Tn$, $\rho_n$ its root vertex, and $(\Ta, \dt, \rho)$ is the stable tree defined in \cref{sctn:stable tree def}. This can also be straightforwardly extended to include convergence of the uniform measure on $\Tn$ to the measure $\Volt$ introduced in \cref{sctn:stable tree def} and therefore give convergence in the pointed Gromov-Hausdorff-Prokhorov topology.

\subsection{Random planar maps}

We now give some basic definitions related to maps.

A \textbf{planar map} is an embedding of a graph $G=(V,E)$ into the sphere in such a way no two edges cross, considered up to orientation-preserving homeomorphism. In what follows we will sometimes drop the adjective ``planar", but all maps should be assumed to planar maps in this paper. A \textbf{rooted} planar map is a planar map equipped with a distinguished oriented edge, called its \textbf{root edge}. The starting vertex of the root edge is called the \textbf{root vertex}.  

A \textbf{pointed} rooted planar map is a rooted planar map that is equipped with an additional distinguished vertex (this may or may not coincide with the root vertex).

We will also need the notion of corners: if we imagine all edges to be replaced by two (oppositely) directed edges, a \textbf{corner} in a map is an angular sector $c$ between an incoming edge and an outgoing edge that follow each other in the clockwise cyclic order around a vertex $v$. The corner is canonically associated with this latter outgoing edge. If $v$ is has degree 1, it is associated with one corner.

\begin{definition} \label{def:quads}
A \textbf{quadrangulation} is a planar map whose faces all have degree four. We denote by $\mathcal{Q}_n$ the set of rooted planar quadrangulations with $n$ faces, and by $\mathcal{Q}_n^{\bullet}$ the set of pointed, rooted planar quadrangulations with $n$ faces.
\end{definition}

A map $\mathfrak{m}$ induces a metric space by considering the set of its vertices $V(\mathfrak{m})$ equipped with the graph distance; the latter is denoted by $d_{\mathfrak{m}}$. For $u \in V(\mathfrak{m})$ and $r \geq 0$, we denote by
\[
B_{\mathfrak{m}}(u,r)=\{v \in V(\mathfrak{m}) \, \lvert \, d_{\mathfrak{m}}(u,v) \leq r\}
\]
the \textbf{ball of radius \boldmath $r$ around $u$ in $\mathfrak{m}$}. A \textbf{geodesic} in $\mathfrak{m}$ is a discrete path $\gamma=(\gamma(i), 0 \leq i \leq k)$ taking values in $V(\mathfrak{m})$, such that $d_{\mathfrak{m}}(\gamma(i),\gamma(j))=\lvert i - j \rvert$ for every $i,j \in \{0, \dotsm, k\}$.

\subsection{The Cori--Vauquelin--Schaeffer bijection}\label{sctn:CVS}

\begin{figure}[htp]
\centering
\includegraphics[scale=0.8]{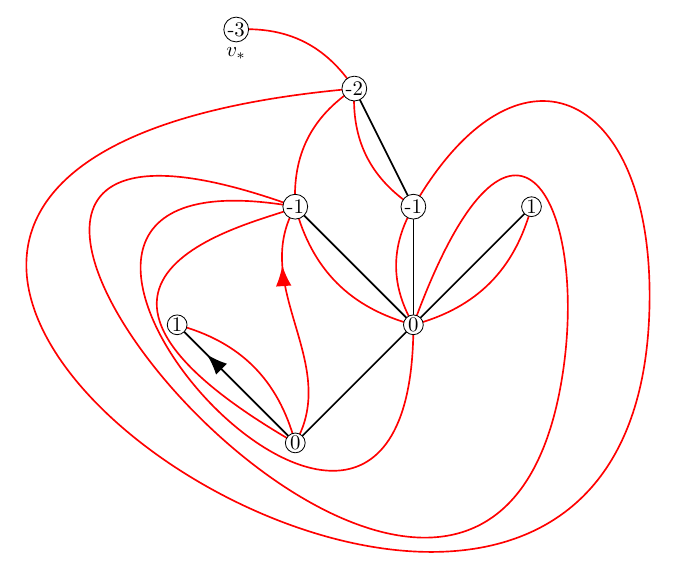}
\caption{The CVS bijection: an admissible labelled tree with 6 edges (in black), and the associated planar quadrangulation with 6 faces (in red), with its pointed vertex $v_*$.}
\label{fig:CVS}
\end{figure}

Let us now get to the bijection between rooted planar quadrangulations with $n$ faces, and a family of decorated trees with $n$ edges, following very closely the presentation of \cite[Section 2.3]{MiermontStFlour2014}. This bijection was first introduced in the work of Cori and Vauquelin \cite{cori1981planar} and further developed in the thesis of Schaeffer \cite{schaeffer1998conjugaison}. We will actually only use the inverse of the pointed version of this bijection, from trees to quadrangulations, which was constructed explicitly by Chassaing and Schaeffer~\cite{chassaing2004random} and refer to it as the \textbf{CVS bijection}.

 Let $\tf$ be a rooted plane tree with $n$ edges, with root edge $e_0$ and root vertex $u_0$. 
An \textbf{admissible label function} on $\tf$ is a function $\ell: V(\tf) \to \ZZ$ such that $\ell(u_0) = 0$, and, for every adjacent $u, v \in V(\tf)$, $| \ell(u) - \ell(v)| \leq 1$. If $\tf$ is a rooted plane tree, and $\ell$ is an admissible label function on $\tf$, we say that the pair $(\tf,\ell)$ is a \textbf{admissible labelled tree}. We denote by $\mathbb{T}_n$ the set of admissible labelled trees with $n$ edges.

Consider a labelled tree $(\tf,\ell) \in \mathbb{T}_n$. Let $e_0 , e_1 , \dots , e_{2n-1}$ be the contour exploration of the oriented edges of $\tf$, as defined in \cref{sctn:discrete tree defs}. Recall that $e^-_i$ is the initial vertex of $e_i$.  It will be useful extend the sequences $(e_i)_i$ and $(e^-_i)_i$ to $\N$ by periodicity. In what follows we will identify the oriented edge $e_i$ with the associated corner, and thus use the notation $\ell(e_i)$ for $\ell(e^-_i)$.

For every $i \geq 0$, we define the \textbf{successor} of $i$ by
\begin{equation}
s(i) := \inf\{j > i \, | \, \ell(e_j) = \ell(e_i) - 1\},
\end{equation}
with the convention that $\inf\varnothing = \infty$. Note that $s(i) = \infty$ if and only if $\ell(e_i) = \min_j \ell(e_j)$, as the values of $\ell$ decrease only by unit steps. We introduce a new vertex, corresponding to a corner labelled $e_{\infty}$, that will play the role of successor for all such $i$. To do so, recall that $\tf$ is embedded in the sphere, and consider a point $v_*$ in the sphere that does not belong to the support of $\tf$. Denote by $e_{\infty}$ a small neighbourhood of $v_*$ not intersecting $\tf$ or any of its corners. We will also refer to $e_{\infty}$ as a corner in what follows. We set
\[
\ell(v_*) = \ell(e_{\infty}) = \min_j \ell(e_j) - 1.
\]
We can then define, for every $0 \leq i < \infty$, the successor of the corner $e_i$ as $s(e_i) = e_{s(i)}$, where we consider $e_{s(i)}$ to be its associated corner.

The reverse CVS construction involves drawing, for every $i \in \{0, 1, . . . , 2n - 1\}$, an edge, that we will call an \textbf{arc}, from the corner $e_i$ to the corner $s(e_i)$, without crossing $\tf$, $v_* $, or another arc (see Figure~\ref{fig:CVS}). This is always possible; see \cite[Lemma 2.3.2]{MiermontStFlour2014}. We denote by $\mathfrak{q}$ the embedded graph with vertex set $V(\tf)\cup{v_*}$, and edge set formed by the arcs.

A first property of $\mathfrak{q}$, which is contained in the stronger statement~\Cref{thm:chassaing-schaeffer}, is that it is a quadrangulation with $n$ faces. Note that it comes equipped with a distinguished vertex $v_*$, but for now it is not a rooted quadrangulation. To fix this root, we use an extra parameter $\epsilon \in \{-1, 1\}$. If $\epsilon = 1$ we let the root edge of $\mathfrak{q}$ be the arc linking $e_0$ with $s(e_0)$, and oriented from $s(e_0)$ to $e_0$. If $\epsilon = -1$, the root edge is this same arc, but oriented from $e_0$ to $s(e_0)$. We have thus defined a mapping $\Phi$ from $\mathbb{T}_n \times \{-1, 1\}$ to the set $\mathcal{Q}^{\bullet}_n$ of pointed rooted planar quadrangulations $(\mathfrak{q}, v_*)$.

\begin{theorem}[Proposition 2 and Theorem 4 in~\cite{chassaing2004random}] For every $n \geq 1$, $\Phi$ is a bijection from $\mathbb{T}_n \times \{-1, 1\}$ onto $\mathcal{Q}^{\bullet}_n$.
\label{thm:chassaing-schaeffer}
\end{theorem}

A very important aspect of the CVS bijection is that the labels of the tree correspond to some of the distances in the quadrangulation. Indeed, if $(\mathfrak{q}, v_*) = \Phi((\tf, \ell), \epsilon)$, then, for every $v \in V(\mathfrak{q})$,
\begin{equation}
d_{\mathfrak{q}}(v, v_*) = \ell(v) - \ell(v_*).
\end{equation}

\begin{remark}\label{rmk:tree paths to map paths}
Another property of the CVS bijection is that every edge in $\tf$ is inscribed inside a face of $\q$. Hence, any path of adjacent vertices in $\gamma_0, \gamma_1, \ldots, \gamma_m$ in $\tf$ can be extended (not necessarily uniquely) to a path in $\q$ by adding at most one extra vertex between each pair of the form $\gamma_i$ and $\gamma_{i+1}$. Hence, if the extended path intersects some set $A \subset V(\q)$, then (the image of) the original path is at distance at most $1$ from $A$.
\end{remark}

For $(\tf,\ell) \in \mathbb{T}_n$, recall that $u_0, u_1, \ldots, u_{n}$ denotes the lexicographical ordering of the tree vertices. Let us define for $0 \leq i,j \leq n$
\begin{equation}
d^{\circ}_{\tf,\ell}(i,j):=\ell (u_i) + \ell(u_j)-2\inf_{i\wedge j \leq k \leq i \vee j}\ell(u_k) +2.
\label{eq:def-d-circ-discrete}
\end{equation}
We also write $d_{\mathfrak{q}}(i,j)$ for $d_{\mathfrak{q}}(u_i,u_j)$. We then have the following bound.

\begin{lemma}[Lemma 3.1 in~\cite{LeGall2007BrownianMapTopological}]\label{lem:discrete-bound-d-dstar}
For every $i,j \in \{0,1, \dotsm, n\}$,
\[
d_{\mathfrak{q}}(i,j) \leq d^{\circ}_{\tf,\ell}(i,j).
\]
\end{lemma}

In the rest of this paper, we will consider random admissible labelled trees $(\Tn, \ell_n)$, constructed as follows. Let $\pof$ be a probability measure on $\ZZ_{\geq 0}$ such that $\pof(0)>0$ and with expectation $1$ and let $Y$ be a real random variable. We assume that 
\begin{equation}
\exists \,c \in (0, \infty) \, : \, \pof[x,\infty) \sim cx^{-\alpha}, \, p_Y(0)<1 \textnormal{ and } p_Y(-1)=p_Y(1)=\frac{1-p_Y(0)}{2}.
\label{eq:hyps-on-mu-and-Y}
\end{equation}
We anticipate that all of our results are adaptable to any law $\pof$ in the domain of attraction of an $\alpha$-stable law, but we focus on the aforementioned case for ease of notation. 

For each $n \geq 2$, let $\Tn$ be a Bienaymé--Galton--Watson tree with offspring distribution $\pof$, conditioned to have $n$ edges (this will be well-defined for all sufficiently large $n$ under our assumptions). Recall that with our conventions, $\Tn$ comes pre-equipped with a root edge, going from the root vertex to its first child in the lexicographical ordering.

Conditionally on $\Tn$, we endow each vertex $u$ of $\Tn$ with a random label $\ell_n(u) \in \R$, setting $\ell_n(\varnothing)=0$ for the root $\varnothing$ of the tree, and, for $v \neq \varnothing$, 
\begin{equation}\label{eqn:ell n def}
\ell_n(v)=\sum_{u \preceq v} Y(u),
\end{equation}
where $(Y(u))_{u \in \Tn\setminus\{\varnothing\}}$ are i.i.d. copies of $Y$. Note that $(\Tn, \ell_n) \in \mathbb{T}_{n}$.

We moreover choose $\epsilon_n$ uniformly in $\{-1, 1\}$, independently for each $n$. The quadrangulation $\Mn$ is defined as the image of the triple $(\Tn, \ell_n, \epsilon_n)$ under the reverse CVS bijection, i.e. $\Phi (\Tn, \ell_n, \epsilon_n)$. We moreover let $\dmn$ denote the graph metric on $\Mn$, let $\partial_n$ denote the extra vertex added when constructing $\Mn$, let $\rho_n$ denote its root vertex, and let $\pi_n$ denote the canonical projection from $\{0, 1, \ldots, n\}$ to $V(\Mn) \setminus \{\partial_n\}$ (via the vertex with that lexicographical index in $\Tn$: it is essentially the identity map). We also let $\nu_n$ denote the uniform probability measure on the vertices of $\Mn$, and let $d^{\circ}_n$ denote $d^{\circ}_{T_n, \ell_n}$ as defined in \eqref{eq:def-d-circ-discrete}.

\begin{remark}
By construction, each vertex in $V(\Mn) \setminus \{\rhoo_n\}$ is incident to at least one edge per corner $c$ that it had in $\Tn$ (this edge is the arc that connects $c$ to its successor), so that the vertex degrees in $\Mn$ are at least as heavy-tailed as in $\Tn$. In that regard, our stable quadrangulations are similar to the images by Tutte's bijection of Boltzmann stable maps, that are planar quadrangulations with stably-decaying vertex degrees, and that should have the same scaling limit as the duals of these Boltzmann maps. In our model however, it is not clear yet if typical vertex degrees are indeed stably-decaying, since each corner of $\Tn$ is not only the origin of a successor arc, but potentially the recipient of many. We plan to address this question in future works.
\label{rem: vertex degrees}
\end{remark}

\subsection{Subsequential scaling limits of random quadrangulations}

The aim of this subsection is to show that the sequence $(\Mn, n^{-\frac{\alpha-1}{2\alpha}} \dmn, \nu_n)$ is tight in the Gromov-Hausdorff-Prokhorov topology. This follows by the same arguments as in the Brownian case, using the main result of \cite{marzouk-brownianstable}.

Let $(\Tn,\ell_n, \epsilon_n)$ be defined as in~\cref{sctn:CVS} (around \eqref{eqn:ell n def}), and recall that $\Mn = \Phi (\Tn, \ell_n, \epsilon_n)$, with root vertex $\rho_n$. On $\Mn$ are defined a distance $\dmn$ and a pseudo-distance $d_n^{\circ}$, and we define these on $\{0, \ldots, n\}^2$ as well using the pushforward by the lexicographical ordering. Let $X_n$ be the Łukasiewicz path of $\Tn$, let $H_n$ be the height process of $\Tn$, both as defined in and below \eqref{eqn:height function def disc}, and extend $H_n$, $X_n$, $\ell_n$, $d_n^{\circ}$ and $\dmn$ to the entire time interval $[0, n]$ and $[0, n]^2$ by linear interpolation. Also let $X^{\text{exc}}$ be as in \cref{sctn:Levy and Ito}. We denote by $H$ the height function defined from $X^{\text{exc}}$ as in \eqref{eqn:height def}, and by $W$ the Brownian snake driven by $H$ as defined in \cref{section:stablesnake}, with tip $\Wt$. The key result is the following.

\begin{theorem}[Special case of Theorem 1.1 in~\cite{marzouk-brownianstable}] \label{thm:marzouk}
Assume that \eqref{eq:hyps-on-mu-and-Y} holds, set $\sigma^2_Y$ to be the variance of $Y$ and set $c' = \left(\frac{c}{C_{\alpha}}\right)^{1/\alpha}$ (recall \eqref{eqn:pi and C def}).

Then the following convergence holds jointly in distribution with respect to the Skorokhod-$J_1$ topology.
\[
\left(\frac{1}{c'}n^{-\frac{1}{\alpha}} X_n(nt), c'n^{-(1-\frac{1}{\alpha})} H_n(nt), \left(\frac{c'n^{-(1-\frac{1}{\alpha})}}{\sigma_Y^2}\right)^{1/2}\ell_n(nt)\right)_{t \in [0,1]} \xrightarrow[n \to \infty ]{(d)} \Big(\X_t, H_t,\Wt_t \Big)_{t \in [0,1]}.
\]
\end{theorem}

Although not stated in \cite[Theorem 1.1]{marzouk-brownianstable}, it is implicit in the proof that we can add the convergence of the Łukasiewicz paths to the above statement (Marzouk works on a probability space where \eqref{eqn:Duquesne contour conv} holds and shows that the snakes converge too).

From the convergence result of~\Cref{thm:marzouk} we derive, similarly to~\cite{LeGall2007BrownianMapTopological}, some preliminary results on subsequential limits for the distances in $\Mn$. Like in the Brownian case, a first straightforward consequence of~\Cref{thm:marzouk} is that
\begin{equation}
\left(\left(\frac{c'n^{-(1-\frac{1}{\alpha})}}{\sigma_Y^2}\right)^{1/2}d^{\circ}_n(ns,nt)\right)_{0 \leq s,t \leq 1}\xrightarrow[n \to \infty]{(d)}\left(D^{\circ}(s,t)\right)_{0 \leq s,t \leq 1},
\end{equation}
where $D^{\circ}$ is defined as in~\eqref{eq:def-Do}. Together with~\Cref{lem:discrete-bound-d-dstar}, this immediately implies the following result, similarly to Proposition 3.2 and its corollary (11) in~\cite{LeGall2007BrownianMapTopological} for the Brownian case.

\begin{proposition}
\label{prop:tightness-metric-spaces}
The sequence of the laws of the processes $\left(n^{-\frac{\alpha-1}{2\alpha}}\dmn(ns,nt)\right)_{0 \leq s,t \leq 1}$ is tight in the space of probability measures on the continuous functions from $[0,1]^2$ to $\R$.

In particular, subsequential limits exist: for any sequence $(n_k)_{k \geq 1}$ there exists a subsequence $(k_m)_{m \geq 1}$ for which there is a limiting function $D$ from $[0,1]^2$ to $\R_+$ along which we have the following joint convergence in distribution:
\begin{align}
\begin{split}\label{eqn:subsequential limits}
&\left(\frac{1}{c'}n^{-\frac{1}{\alpha}}X_n(nt), c'n^{-\frac{\alpha-1}{\alpha}}H_n(nt),\left(\frac{c'n^{-(1-\frac{1}{\alpha})}}{\sigma_Y^2}\right)^{1/2}\ell_n(nt),\left(\frac{c'n^{-(1-\frac{1}{\alpha})}}{\sigma_Y^2}\right)^{1/2}\dmn(ns,nt) \right)_{s,t \in [0,1]}\\
&\qquad \xrightarrow[n \to \infty]{(d)}\left(\X_t, H_t, \Wt_t, D(s,t)\right)_{s,t \in [0,1]}.
\end{split}
\end{align}
\end{proposition}
\begin{proof}
The first statement has the same proof as \cite[Proposition 3.2]{LeGall2007BrownianMapTopological}. The second statement is the analogue of \cite[Equation (11)]{LeGall2007BrownianMapTopological} and comes from combining the first statement with \cref{thm:marzouk}.
\end{proof}

It is important to note that a priori, the limiting function $D$ may depend on the subsequence. In what follows, we only consider limits along such a subsequence $(n_k)_{k \geq 1}$. By the Skorokhod representation theorem, we can further assume that the convergence~\eqref{eqn:subsequential limits} holds almost surely. This has the following consequence, that is similar the first assertion of~\cite[Theorem 3.4]{LeGall2007BrownianMapTopological} in the Brownian case.

\begin{corollary}
The sequence of the laws of the metric spaces $(V(\Mn),n^{-\frac{\alpha-1}{2\alpha}}\dmn, \nu_n, \rho_n)$ is tight in the space of probability measures on the space $\mathbb{F}^c$ of isometry classes of compact rooted metric spaces equipped with a finite Borel measure, endowed with the pointed Gromov-Hausdorff-Prokhorov topology. In particular, for any sequence $(n_k)_{k \geq 1}$ there exists a subsequence $(k_m)_{m \geq 1}$ for which there exists a limiting compact rooted metric-measure space $(\M, D, \nu, \rho)$ such that
\begin{align}
\begin{split}\label{eqn:subsequential limits m spaces}
\left(V(M^{(\alpha)}_{n_{k_m}}),\frac{\sqrt{c'}}{\sigma_Y}n_{k_m}^{-\frac{\alpha-1}{2\alpha}}d_{n_{k_m}}, \nu_{n_{k_m}}, \rho_{n_{k_m}}\right)\xrightarrow[n \to \infty]{(d)} (\M, D, \nu, \rootm).
\end{split}
\end{align}
Moreover, $\M$ is a quotient space of $\Ma=\T/\approx$, as defined after~\eqref{eq:def-D*}, and in particular there is a canonical projection $\pi: [0,1] \to \M$. The limiting metric is the pushforward of the function $D$ appearing in \eqref{eqn:subsequential limits} under $\pi$, the limiting measure $\nu$ is the pushforward of Lebesgue measure under $\pi$, and the root $\rootm$ is the pushforward of $0$ under $\pi$.
\end{corollary}
\begin{proof}
Let us define $C' =\sqrt{c'}/\sigma_Y$ for ease of notation. Throughout the proof we will consider the space $(V(\Mn)\setminus \{ \partial_n \} ,C'n^{-\frac{\alpha-1}{2\alpha}}\dmn, \nu_n, \rho_n)$ in place of $(V(\Mn),C'n^{-\frac{\alpha-1}{2\alpha}}\dmn, \nu_n, \rho_n)$ and for the rest of the proof let $\nu_n$ denote uniform measure on $V(\Mn)\setminus \{\partial_n\}$; by the triangle inequality, the two spaces necessarily have the same scaling limit.
 
Firstly, note that \eqref{eqn:subsequential limits} implies that the sequence $(V(\Mn)\setminus \{\partial_n\},C'n^{-\frac{\alpha-1}{2\alpha}}\dmn)$ is tight in the Gromov-Hausdorff topology by \cite[Theorem 7.4.15]{burago-burago-ivanov}. By the definitions of the various Gromov-Hausdorff and Gromov-Hausdorff-Prokhorov distances, it immediately follows that the spaces are also compact with respect to the pointed Gromov-Hausdorff-Prokhorov topology on adding the roots and measures. Since the space $\mathbb{F}^c$ is Polish (this follows from a very minor extension of \cite[Proposition 8]{miermont2009tessellations}), this implies that there are subsequential limits taking values in $\mathbb{F}^c$.

We now work on a probability space and along a subsequence where \eqref{eqn:subsequential limits} holds almost surely. Our candidate for the limiting metric space is $([0,1]/\sim_D, D)$, where $s \sim_D t \iff D(s,t) = 0$. Note that $([0,1]/\sim_D, D)$ is indeed a compact metric space since all of the relevant properties are inherited directly from \eqref{eqn:subsequential limits}. We write $M=[0,1]/\sim_D$, with root $\{0\}$, and let $\pi$ denote the canonical projection from $[0,1]$ to $\M$.

We first show that the pointed Gromov-Hausdorff distance between $(V(\Mn)\setminus \{\partial_n\},C'n^{-\frac{\alpha-1}{2\alpha}}\dmn, \rho_n)$ and $(\M, D, \{0\})$ vanishes almost surely using \cref{lem:distortion}. In particular, since $\Mn$ is defined as the image of $(\Tn, \ell_n, \epsilon_n)$ under $\Phi$, it follows that $(V(\Mn)\setminus \{\partial_n\},C'n^{-\frac{\alpha-1}{2\alpha}}\dmn, \rho_n)$ is isometric to
\[
\left([0,1] / \sim_n , C'n^{-\frac{\alpha-1}{2\alpha}}\dmn (\lfloor sn \rfloor, \lfloor tn \rfloor), 0 \right),
\]
where $\sim_n$ denotes the equivalence relation $s \sim_n t \iff \dmn \left(\lfloor sn \rfloor, \lfloor tn \rfloor \right)=0$. Recall from Section \ref{sctn:CVS} that the two spaces are actually related via the projection $s \mapsto \pi_n(\lfloor sn \rfloor) \in V(\Mn)\setminus \{ \partial_n \}$.

We construct a correspondence $\mathcal{R}_n$ between $(V(\Mn)\setminus \{ \partial_n \},C'n^{-\frac{\alpha-1}{2\alpha}}\dmn, \rho_n)$ and $(\M,D,\{0\})$ by including all points of the form $(\pi_n(\lfloor ns \rfloor), \pi(s))$ where $s \in [0,1]$. We denote its distortion by $r_n$; it then follows directly from \eqref{eqn:subsequential limits} that $r_n \to 0$ along the same subsequence.

To prove that the measures also converge on this space, we take the Gromov-Hausdorff embedding $F_n = (V(\Mn)\setminus \{ \partial_n \}) \sqcup \M$ endowed with the metric
\[
D_{F_n}(x,y) = \begin{cases}
C'n^{-\frac{\alpha-1}{2\alpha}}\dmn(x,y) & \text{ if } x, y \in V(\Mn)\setminus \{ \partial_n \}\\
D(x,y) & \text{ if } x, y \in \M\\
\inf_{u, v \in \mathcal{R}_n} (C'n^{-\frac{\alpha-1}{2\alpha}}\dmn(x,u) + D(y,v) + \frac{1}{2} r_n) & \text{ if } x \in V(\Mn)\setminus \{ \partial_n \}, y \in \M.
\end{cases}
\]

We now claim that $d^{F_n}_P(\nu_n, \nu) \rightarrow 0$ as $n \rightarrow \infty$. To prove this, for each $0 \leq i \leq n$ we define the interval $I_{n,i} = [\frac{i}{n}, \frac{i}{n} + \frac{1}{n})$. Now take a set $A_n$ of vertices in $\Mn \setminus \{\partial_n\}$, and set
\[
A'_n = \bigcup_{i: \pi_n(i) \in A_n} I_{n,i}.
\]
Let $A_n'' = \pi(A'_n)$. We will show that $A_n'' \subset A_n^{r_n}$, where $E^{\epsilon}$ denotes the $\epsilon$-neighborhood of a set $E$. To this end, take any $v \in A_n''$. By our definitions, there exists $s \in A'_n$ with $v=\pi(s)$ and $s \in I_{n,i}$ for some $i$ with $\pi_n(i) \in A_n$. Consequently, we must have that $i = \lfloor ns \rfloor$, and hence $(u_i, v) \in \mathcal{R}_n$ and $D_F(u_i,v) = \frac{1}{2} r_n$. It follows that $v \in A_n^{r_n}$ and hence that $A_n'' \subset A_n^{r_n}$.

Also note that $\nu_n(A_n) = \nu(A_n'')$ by construction, and so $\nu_n(A_n) \leq \nu(A_n^{r_n})$.

Similarly, take any set $B \subset \M$. We use the same argument to show that $\nu(B) \leq \nu_n(B^{r_n})$. Let $B' = \pi^{-1}(B)$, and 
\[
B_n = \{ \pi_n(i) \colon \exists \ s \in B' \text{ with } s \in I_{n,i} \}.
\]
Clearly $B' \subset \bigcup_{i: \pi_n(i) \in B_n} I_{n,i}$ and so
\begin{equation}\label{eqn:Prokh convergence inclusion}
\nu(B) = Leb(B') \leq Leb\left( \bigcup_{i: \pi_n(i) \in B_n} I_{n,i}\right) = \frac{|B_n|}{n} = \nu_n(B_n).
\end{equation}
To conclude, note that if $v \in B_n$ and $v = \pi_n(i)$, then (by our definitions) there necessarily exists $s \in B'$ with $s \in I_{n,i}$, hence $i = \lfloor sn \rfloor$ and therefore $(v, \pi(s)) \in \mathcal{R}_n$. Moreover $\pi(s) \in B$. Hence $B_n \subset B^{r_n}$, so $\nu_n(B_n) \leq \nu_n(B^{r_n})$, so combining with \eqref{eqn:Prokh convergence inclusion} gives $\nu(B) \leq \nu_n(B^{r_n})$.

It follows that $d^{F_n}_P(\nu_n, \nu) \leq r_n$, and hence converges to zero as $n \rightarrow \infty$.

Moreover, it follows from~\Cref{thm:marzouk} and from the construction of $\Da^*$ that $D \leq \Da^*$, and therefore that $(\M,D, \{0\})$ is a quotient space of $(\T/\approx, \Da^*, \rho)$. Since the roots are the same we will henceforth identify them.
\end{proof}

\begin{remark}
Note that each limiting space $(M, D, \nu, \rho)$ is naturally defined under $\Noo_0$. However, thanks to points (i), (ii) and (iii) just above \eqref{eqn:Ito measure integrate s}, it can equally be defined under $\N^{(s)}$ for any $s>0$ or simply under $\N_0$.
\end{remark}

We stress again that the limiting space $(\M, D, \nu, \rootm)$ a priori depends on the subsequence. Similarly to the Brownian case, the construction allows us to immediately deduce some properties for $D$. This follows by the same proof as \cite[Proposition 3.3]{LeGall2007BrownianMapTopological}.

\begin{proposition} 
\label{prop:properties of D}
The following properties hold almost surely.
\begin {enumerate}[(1)]
\item The function $D$ is a pseudo-distance on $[0,1]$, i.e., for every $s,t,u \in [0,1]$
\begin{itemize}
\item $D(s,s)=0$
\item $D(s,t)=D(t,s)$
\item $D(s,u)\leq D(s,t) +D(t,u)$.
\end{itemize}
\item For every $s,t \in [0,1]$, $D(s,t) \leq \Da^*(s,t) \leq \Da^{\circ}(s,t)$ (recall the definitions from~\eqref{eq:def-Do} and~\eqref{eq:def-D*}).
\item Let $s^*$ denote the location of the unique (see \cref{prop:minimum unique}) infimum of $\Wt$. Then for every $s \in [0,1]$, $D(s^*, s) = \Wt_s-\Wt_{s^*}$. Consequently, for any $r\geq0$, $ \nu (B(\pi(s^*), r)) = \overline{\I}[0, r] $ (recall the definition of $\overline{\I}$ in \eqref{eqn:ISE def rerooted}).
\end{enumerate}
\end{proposition}

We end with a useful fact. Recall that $[[u,v]]$ is the direct tree path from $u$ to $v$; the following is morally due to the fact that the geodesic cannot take ``shortcuts" to cut out the direct path.

\begin{lemma}\label{lem:D snake upper bound}
For any $u, v \in \M$,
\begin{align*}
D(u,v) \geq Z_u + Z_v - 2 \min_{x \in [[u,v]]} Z_x.
\end{align*}
\end{lemma}
This can be proved formally using \cite[Proposition 2.3.8(ii)]{MiermontStFlour2014}, which gives the discrete analogue of this statement (and is deterministic), and taking appropriate limits.
Alternatively, one could also repeat the proof leading up to \cite[Corollary 3.2]{legall-scalinglim-geod}, which gives the analogous result in the Brownian case, and for which the key input is the fact that the tree and the snake do not have common increase points (in the stable case, this latter statement was established in \cite[Theorem 1.4]{archer2024snakes}). There is also an another viewpoint in terms of cactuses: see \cite[Remark 3.3]{curien2013brownian} and the associated discussion.

\section{Large volumes}\label{sctn:uniform points large}
The purpose of this section is to obtain some estimates on the probability of having unusually large volumes at typical points in $(\M, D, \nu, \rho)$. Throughout this section, $U$ will denote a uniform random variable on either $[0, \sigma]$ or $[0,1]$ (under $\N_0$ or $\Noo_0$ as appropriate), chosen independently of everything else, $p(U)$ will denote its image in $\Ta$, and $\pi(U)$ will denote its image in the limiting space $\M$.

The upper bounds on the corresponding tail probabilities are very straightforward; the lower bounds are more involved.

\subsection{Upper tail bounds for large volumes}\label{sctn:UB on uniform balls}

In this subsection we fix an $a > 0$ and consider the law of $M_{\epsilon}(U_a)$, where $M_{\epsilon}(s) := \epsilon^{-\frac{2\alpha}{\alpha-1}} \nu(B(s, \epsilon))$ and where $U_a$ is chosen ``uniformly at level $a$ in the tree". Formally, we sample a pointed rooted tree and snake $(\Ta, U_a, Z)$ according to $\N_0 \times \ell^a$, where the local time measure $\ell^a$ is as in \eqref{eqn:local a def}. Note that this defines a probability measure on pointed rooted trees since $\langle \ell^a, 1\rangle = 1$ by \cref{prop:firstmomentdLs}.

\begin{proposition}\label{prop:upper vol tail bound lazy}
There exists $c < \infty$ such that for all $a \in (0, \infty)$,  $r>0$ and $\lambda >1$,
\[
\N_0 \left( \nu (B(U_a,r)) \geq \lambda r^{\frac{2\alpha}{\alpha-1}} \right) \leq c\lambda^{-\frac{\alpha-1}{2\alpha}}.
\]
\end{proposition}
\begin{remark}
We do not believe that the exponent for the tail decay is optimal.
\end{remark}
\begin{proof}
The proof is similar to that of \cite[Proposition 5.1]{duquesne2005hausdorff}. We apply Proposition \ref{prop:firstmomentdLs} and use the same notation, except that it will be more convenient to consider the linear Brownian motion $(\xi_t)_{t \geq 0}$ as starting at $U_a$ rather than at the root. Let $\tau_x$ denote the hitting time of $x$ by $(\xi_t)_{t \geq 0}$. It follows from Proposition \ref{prop:Poisson master formula}, Proposition \ref{prop:firstmomentdLs} and Lemma \ref{lem:D snake upper bound} that, for any $\delta>0$,
\begin{align*}
\N_0 \left( \int \ell^a (ds) e^{-\delta M_{\epsilon}(s)}\right) &\geq \Pi_0 \left( \exp \left\{ -\alpha \int_0^{\tau_{-\epsilon} \wedge a} \N_0 (1-e^{-\delta \epsilon^{-\frac{2\alpha}{\alpha-1}} \sigma})^{\alpha-1} dt  \right\} \right) \\
&\geq \Pi_0 \left( \exp \left\{ -\alpha \int_0^{\epsilon^{-2} \tau_{-\epsilon}} \epsilon^2 \N_0 (1-e^{-\delta \epsilon^{-\frac{2\alpha}{\alpha-1}} \sigma})^{\alpha-1} dt  \right\} \right) \\
&= \Pi_0 \left( \exp \left\{ -\alpha \int_0^{\tau_{-1}} \N_0 (1-e^{-\delta \sigma})^{\alpha-1} dt  \right\} \right).
\end{align*}
Here in the second and third line we used Brownian scaling and the \Ito scaling property \eqref{eqn:Ito scaling}.

Note that $\tau_{-1}$ has density function  
\[
\frac{1}{\sqrt{4\pi x^3}}e^{-1/4x} dx,
\]
and so a Tauberian theorem implies that there exists $c>0$ such that $1-\Pi_0 \left( \exp \left\{ -\theta \tau_{-1}\right\}\right) \sim c \theta^{1/2}$ as $\theta \downarrow 0$. Hence by \eqref{eqn:Laplace sigma} we can write for all sufficiently small $\delta>0$:
\begin{align*}
\Pi_0 \left( \exp \left\{ -\alpha \int_0^{\tau_{-1}} \N_0 (1-e^{-\delta \sigma})^{\alpha-1} dt  \right\} \right) &= \Pi_0 \left( \exp \left\{ -\alpha  \delta^{\frac{\alpha-1}{\alpha}} \cdot {\tau_{-1}}   \right\} \right) \geq 1 - c\delta^{\frac{\alpha-1}{2\alpha}}.
\end{align*}

Recall from below \cref{prop:firstmomentdLs} that $\N_0 \left( \int \ell^a (ds)\right)=1$. Hence applying Markov's inequality we deduce that 
\begin{align*}
\N_0 \left( \int \ell^a (ds) \mathbbm{1}\{ M_{\epsilon}(s) \geq \lambda\} \right) &\leq (1-e^{-1})^{-1}\N_0 \left( \int \ell^a (ds) (1 - e^{-\lambda^{-1} M_{\epsilon}(s)})\right) \\
& \leq c(1-e^{-1})^{-1}\lambda^{-\frac{\alpha-1}{2\alpha}} . \qedhere
\end{align*}
\end{proof}

This allows us to deduce the following. 

\begin{corollary}\label{cor:log volume fluctuations UB}
Take any $\epsilon>0$. Then $\N_0$-almost surely, it holds for $\nu$-almost every $x$ in $\M$ that
\[
\limsup_{r \downarrow 0} \frac{\nu (B(x,r))}{r^{\frac{2\alpha}{\alpha-1}}(\log r^{-1})^{\frac{2\alpha+\epsilon}{\alpha-1}}} = 0.
\]
\end{corollary}
\begin{proof}
We start by fixing an arbitrary $a>0$ and proving the result for $\ell^a$-almost every $x$; this then extends straightforwardly to $\nu$-almost every $x$ using \eqref{eqn:local a def}. By Fubini's theorem, it is sufficient to prove the claim for $B(U_a,r)$, where $U_a$ is chosen according to $\ell^a$. This latter statement follows quite straightforwardly from \cref{prop:upper vol tail bound lazy}. In particular, combining the tail estimate of \cref{prop:upper vol tail bound lazy} with Borel-Cantelli along the subsequence $r_k = 2^{-k}$ and using monotonicity of volumes, we deduce that, $\N_0$-almost everywhere
\begin{equation*}
\limsup_{r \downarrow 0} \frac{\nu (B(x,r))}{r^{\frac{2\alpha}{\alpha-1}}(\log r^{-1})^{\frac{2\alpha+\epsilon}{\alpha-1}}} \leq 2^{\frac{2\alpha}{\alpha-1}} \limsup_{k \to \infty} \frac{\nu (B(x,r_k))}{r_k^{\frac{2\alpha}{\alpha-1}}(\log r_k^{-1})^{\frac{2\alpha+\epsilon}{\alpha-1}}} = 0, \text{  for } \ell^a\text{-almost every } x.
\end{equation*}
 Moreover by considering the restriction to $\{1 < \sigma < 2\}$ and rescaling, we deduce the same result under $\Noo_0$.
\end{proof}

\subsection{Lower tail bounds for large volumes}

We now turn to establishing a similar lower bound. We anticipate that the exponent $\alpha-1$ appearing here is sharp.

\begin{proposition}\label{prop:lower vol tail bound}
\begin{enumerate}[(a)]
\item There exists $c > 0$ such that for all $r>0, 1< \lambda < r^{-\frac{\alpha}{\alpha-1}}$,
\begin{align*}
\Noo_0 \left( \nu (B(\pi(U),r)) \geq \lambda r^{\frac{2\alpha}{\alpha-1}} \right) &\geq c\lambda^{-(\alpha-1)}.
\end{align*}
\item There exists $c > 0$ such that for all $r>0, 1< \lambda < \infty$,
\begin{align*}
\N_0 \left( \nu (B(\pi(U),r)) \in [\lambda r^{\frac{2\alpha}{\alpha-1}}, 2\lambda r^{\frac{2\alpha}{\alpha-1}}] \right) &\geq c\lambda^{-(\alpha-1)}r^{-\frac{2}{\alpha-1}}.
\end{align*}
\end{enumerate}
\end{proposition}

\begin{remark}
The bound in part (a) should really hold for all $1< \lambda \ll r^{-\frac{2\alpha}{\alpha-1}}$. The reason for the considering the interval $[\lambda r^{\frac{2\alpha}{\alpha-1}}, 2\lambda r^{\frac{2\alpha}{\alpha-1}}]$ in part (b) will become apparent in the proof of Corollary \ref{cor:log volume fluctuations sup LB}, where it will enable us to consider disjoint events for a decreasing sequence of radii.
\end{remark}

Comparing with the following proposition for the volume of the ball around the minimum of the snake (recall this is denoted $\pi(s^*)$) leads to a proof of Corollary \ref{cor:minimum atypical}. This contrasts with the Brownian case, where $\pi (s^*)$ has the same law as a vertex chosen according to $\nu$.

\begin{proposition}\label{prop:expectation minimum}
It holds that
\begin{equation}
\limsup_{\varepsilon \downarrow 0} \varepsilon^{-\frac{2 \alpha}{\alpha -1}} \Noo_0 \left( \nu (B(\pi(s^*), \epsilon)) \right) < \infty.
\end{equation}
\end{proposition}
\begin{proof}
This is a straightforward consequence of \cite[Theorem 1.3]{archer2024snakes} and~\cref{prop:properties of D} (3) (to obtain the result under $\Noo_0$ we can restrict to the event $\{1<\sigma<2\}$ and apply rescaling).
\end{proof}

\begin{proof}[Proof of Corollary \ref{cor:minimum atypical}]
If $\pi(s^*)$ did have the same law as $\pi(U)$ for an independently chosen $U$, then Markov's inequality combined with Proposition \ref{prop:expectation minimum} would imply that there exists a constant $C<\infty$ such that 
for all $r>0, 1< \lambda < r^{-\frac{\alpha}{\alpha-1}}$,
\[
\Noo_0 \left( \nu (B(\pi(U),r)) \geq \lambda r^{\frac{2\alpha}{\alpha-1}} \right) \leq C\lambda^{-1}.
\]
Taking for example $\lambda = \log r^{-1}$ we find that this result is incompatible with \cref{prop:lower vol tail bound}(a) for all sufficiently small $r$.
\end{proof}

Similarly to the previous subsection, this has consequences for volume fluctuations around typical points in $\M$.

\begin{corollary}\label{cor:log volume fluctuations sup LB}
Take any $\epsilon>0$. Then $\N_0$-almost everywhere, it holds for $\nu$-almost every $x$ in $\M$ that
\[
\limsup_{r \downarrow 0} \frac{\nu (B(x,r))}{r^{\frac{2\alpha}{\alpha-1}}(\log r^{-1})^{\frac{1-\epsilon}{\alpha-1}}} = \infty.
\]
\end{corollary}

We now turn to the proof of Proposition \ref{prop:lower vol tail bound}, which will proceed by a series of lemmas. The strategy is to consider an event on which $U$ is close to a large hub, which in turn has a correspondingly large volume.

Recall that $U$ is a uniform point in $[0, \sigma]$; we let $H_U$ denote the height of $p(U)$ in $\Ta$. We consider the spinal decomposition along the spine from $\rho$ to $p(U)$ as in Section \ref{sctn:spinal decomp} and recall that the sizes of hubs along this spine are determined by two subordinators $U^{(1)}$ and $U^{(2)}$: one for each side of the path. We let $I_r$ denote the part of this spine that is within tree distance $r^2$ of $p(U)$, and consider the sizes of the parts of the hubs that are on the right hand side of this part of the spine (i.e. the side that comes after $p(U)$ in the contour exploration). Wlog we suppose that $U^{(1)}$ corresponds to the right hand side; each jump of $U^{(1)}$ then corresponds to a ``half-hub'' of the same size on the spine. We consider the event $A_r = A_{r, \lambda}$ (we drop the $\lambda$ to lighten notation) that in $I_r$ there exists a half-hub of size exceeding $\lambda r^{\frac{2}{\alpha-1}}$, and that the sum of the sizes of all of the rest of the half-hubs in this interval is at most $r^{\frac{2}{\alpha-1}}$.

\begin{figure}[h]
\begin{center}
\includegraphics[height=1.5cm]{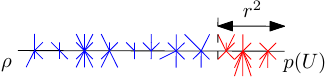}
\caption{Schematic representation of the sizes of hubs (and half-hubs) on each side of the spine from $\rho$ to $p(U)$. There are subtrees attached according to the Poisson process described in \cref{sctn:spinal decomp}. The red and blue parts can be treated independently.}\label{fig:red blue spine}
\end{center}
\end{figure}

We also define $\tilde{\sigma}_r$ to be the sum of the lifetimes of all the subtrees grafted to the spine outside of the interval $I_r$ (i.e. the sum of the volumes of all of the subtrees in blue in Figure \ref{fig:red blue spine}). The quantity $\tilde{\sigma}_r$ will be useful because it is a good approximation for $\sigma$, but it can be treated independently of the red subtrees in calculations.

\begin{lemma}\label{lem:Ar prob LB}
There exists $c > 0$ such that for all $r>0, 1< \lambda \ll r^{-\frac{2}{\alpha-1}}$,
\[
\N_0 (A_r, 1 < \tilde{\sigma}_r < 2) \geq c\lambda^{-(\alpha-1)}.
\]
\end{lemma}
\begin{proof}

We extend and abuse our notation slightly by writing $A_r(s)$ for the event $A_r$ conditionally on the event $U=s$, and by writing $A_r(U^{(1)}_{[0, r^2]})$ for the relevant event in terms of $U^{(1)}$ on the interval $[0, r^2]$ where $U^{(1)}$ is the subordinator discussed above.

We will apply \cref{prop:firstmomentrange ds} and lower bound the right hand side by considering only the contribution from $h \in [1,2]$. We take the setup described above this lemma, and for convenience we will index everything by distance from $p(U)$, so that a jump of $U^{(1)}$ at time $t$ corresponds to a half-hub on the right hand side of the path at distance $t$ from $p(U)$.

In addition, we will let $\mathcal{P}_1$ denote the Poisson process of trees grafted to the spine within distance $r^2$ of $p(U)$ (in red in Figure \ref{fig:red blue spine}), and $\mathcal{P}_2$ denote the Poisson process of trees grafted everywhere else (in blue in Figure \ref{fig:red blue spine}). These are independent. Similarly the jumps of $U^{(1)}$ in $I_r$ are independent of the jumps of $U^{(1)}+U^{(2)}$ outside of $I_r$. By Proposition \ref{prop:firstmomentrange ds}, this allows us to factorise as follows:
\begin{align*}
&\N_x \left[ \frac{\mathbbm{1}\{1 < \tilde{\sigma}_r < 2\}}{\sigma} \int_0^{\sigma} A_r(s) ds \right] \\
&\qquad \geq \frac{1}{2} \int_1^2 dh E^0 \left[ A_r(U^{(1)}_{[0, r^2]}), \sigma < 2, 1 < \tilde{\sigma}_r < 2 \right] \\
&\qquad \geq \frac{1}{2} \int_1^2 dh E^0 \left[ A_r(U^{(1)}_{[0, r^2]}), \sigma - \tilde{\sigma}_r \leq \frac{1}{2} \right] E^0\left[ \mathbb P_{\mathcal{P}_2} \left( 1 < \tilde{\sigma}_r < \frac{3}{2} \right)\right] \\
&\qquad \geq \frac{1}{2} \int_1^2 dh \left\{ P^0 \left( A_r(U^{(1)}_{[0, r^2]}) \right) -  E^0\left[ \mathbb P_{\mathcal{P}_1} \left( \sigma - \tilde{\sigma}_r > \frac{1}{2} \right) \right] \right\} E^0\left[ \mathbb P_{\mathcal{P}_2} \left( 1 < \tilde{\sigma}_r < \frac{3}{2} \right)\right].
\end{align*}
We start by bounding $P^0 \left( A_r(U^{(1)}_{[0, r^2]})  \right)$ (this takes up most of the proof). We assume for now that $\lambda>2$ and separate the jumps of $U^{(1)}_{[0,r^2]}$ into three groups:
\begin{itemize}
\item those of size less than $r^{\frac{2}{\alpha-1}}$,
\item those of size more than $\lambda r^{\frac{2}{\alpha-1}}$, 
\item those of size in $[r^{\frac{2}{\alpha-1}}, \lambda r^{\frac{2}{\alpha-1}}]$.
\end{itemize}
By standard properties of subordinators and Poisson processes, these three collections of jumps can be considered independently. The sum of the jumps of $U^{(1)}$ in $I_r$ with size at most $r^{\frac{2}{\alpha-1}}$ is stochastically dominated by an independent copy of $U^{(1)}_{r^2}$; hence there is a constant probability that the sum is at most $r^{\frac{2}{\alpha-1}}$ using scaling invariance of $U^{(1)}$ (which implies that the probability is the same for all $r$). The number of jumps of size exceeding $xr^{\frac{2}{\alpha-1}}$ is \textsf{Poisson}($C_{\alpha}x^{-(\alpha-1)}$), where $C_{\alpha}$ is as in \eqref{eqn:pi and C def}; hence for all $\lambda > 2$ the probability that there are zero with size in $[r^{\frac{2}{\alpha-1}}, \lambda r^{\frac{2}{\alpha-1}}]$ and exactly one with size exceeding $\lambda r^{\frac{2}{\alpha-1}}$ is lower bounded by $e^{-2C_{\alpha}}C_{\alpha}\lambda^{-(\alpha-1)}$. We obtain the desired result by multiplying these three probabilities. This implies that $P^0 \left( A_r(U^{(1)}_{[0, r^2]})  \right) \geq c\lambda^{-(\alpha-1)}$ for all $\lambda>2$ and extends to all $\lambda>1$ by monotonicity.

It is then straightforward to verify that $E^0 \left[\mathbb P_{\mathcal{P}_2} \left( 1 < \tilde{\sigma}_r < \frac{3}{2} \right)\right]$ is lower bounded by positive quantities which don't depend on $r$ (whenever $r$ is sufficiently small), and that (by \eqref{eqn:N lifetime tails}, standard results on tails of stable processes and our assumption that $\lambda \ll r^{-\frac{2}{\alpha-1}}$) there exists $c>0$ such that
\[
 E^0\left[ \mathbb P_{\mathcal{P}_1} \left( \sigma - \tilde{\sigma}_r > \frac{1}{2} \right) \right] \leq cr^{-2} \ll c \lambda^{-(\alpha-1)},
\]
thus concluding the proof of the lemma.
\end{proof}

According to the formula in \cref{prop:firstmomentrange ds}, conditionally on $U^{(1)}$, the subtrees grafted to the right hand side of the spine in the interval $I_r$ have the law of a Poisson process with intensity
\begin{equation}
\mathbbm{1}\{t \in [0,r^2]\} dU^{(1)}_t \otimes \mathbb N_{\xi_t} (dw),
\end{equation}
where $(\xi_t)_{t \geq 0}$ is the Brownian motion appearing in  \cref{prop:firstmomentrange ds}; in other words the Brownian motion on the spine indexed by distance from $\pi(U)$.

Conditionally on the event $A_r$, we now consider the following event $B_r=B_{r, \lambda}$ (again we drop the dependence on $\lambda$ to ease notation). Let $h = h_{r, \lambda}$ denote the hub of size at least $\lambda r^{\frac{2}{\alpha-1}}$. Throughout the rest of the proof, when we consider the values of the snake restricted to a subtree, we will retain its original values (i.e. we will \textit{not} re-root the snake to start at $0$).

The event $B_r$ is then defined as follows (recall that we parametrised time so that $\xi_0$ is the value of the snake at $p(U)$, and that $\mathcal{R}$ denotes the range of the snake):
\begin{itemize}
\item The values of the snake in $I_r$ are all in the interval $[\xi_0-\frac{r}{16}, \xi_0+\frac{r}{16}]$.% for all $t \in [0, r^2]$.
\item All subtrees grafted to a hub $h'$ in $[0, r^2]$ that is not $h$ satisfy $\mathcal{R} (Z) \subset [ \xi_{h'} - \frac{r}{16},  \xi_{h'} + \frac{r}{16}]$.
\end{itemize}

We then have the following. 
\begin{lemma}\label{lem:Br given Ar}
There exists $c' > 0$ such that for all $r>0, 1< \lambda < \infty$,
\[
\N_0 (B_r | A_r, 1< \tilde{\sigma}_r < 2) \geq c'.
\]
\end{lemma}
\begin{proof}

Note that the two events appearing in the bullet points above can be treated independently of each other and also independently of the event $\{1< \tilde{\sigma}_r<2\}$ since the Poisson process representation holds conditionally on $\xi$, and since different parts of the Poisson process are independent. The first event has uniformly positive probability by the scaling property of Brownian motion. Since the total mass of all the other hubs excluding $h$ is at most $r^{\frac{2}{\alpha-1}}$ on the event $A_r$, the number of subtrees violating the second bullet point is stochastically dominated by a Poisson random variable with mean $2 \left( 16^2\frac{\alpha + 1}{(\alpha - 1)^2} \right)^{\frac{1}{\alpha -1}}$ (also using \eqref{eq:rangesnake}), and hence is also zero with positive probability.

In summary, we obtained a uniformly positive probability for all both of the desired events, so the result follows on multiplying these (and adjusting the resulting constant lower bound to extend to all $\lambda>1$).
\end{proof}

On the event $A_r$ we now let $V_r = V_{r, \lambda}$ denote the collection of subtrees grafted to $h$ in which the snake does not exit $[\xi_h-r/8, \xi_h+r/8]$.

\begin{lemma}\label{lem:vol set LB cond}
Let $c'$ be the constant appearing in \cref{lem:Br given Ar}. Then there exists $\tilde{c}>0$ such that, for all $r>0, 2\tilde{c}< \lambda < \infty$,
\[
\N_0(\nu (V_r) \geq \tilde{c}\lambda r^{\frac{2\alpha}{\alpha-1}}|A_r, 1< \tilde{\sigma}_r < 2) \geq 1-\frac{c'}{2}.
\]
\end{lemma}
\begin{proof}
First note that we can drop the dependence on $\tilde{\sigma}_r$ since this is independent of all variables that only depend on the trees grafted to $I_r$, and in particular $(V_r,\indic{A_r})$. By \cref{prop:firstmomentrange ds}, \cref{prop:Poisson master formula} and the scaling property \eqref{eqn:Ito scaling}, for any $\kappa>0$ we can then write
\begin{align*}
\N_0 \left( e^{-\kappa V_r} |A_r \right) &\leq \exp \left\{-\lambda r^{\frac{2}{\alpha-1}} \N_0 \left( (1-e^{-\kappa \sigma}) \mathbbm{1}\left\{ \frac{r}{8}, \frac{-r}{8} \notin \mathcal{R}\right\} \right)\right\} \\
&\leq \exp \left\{-\lambda (\kappa r^{\frac{2\alpha}{\alpha-1}} - \kappa^2 r^{\frac{4\alpha}{\alpha-1}})  \N_0 \left(1 \leq \sigma \leq 2, \frac{1}{8},  \frac{-1}{8} \notin \mathcal{R} \right) \right\}.
\end{align*}
Hence setting $\kappa =c \lambda^{-1} r^{-\frac{2\alpha}{\alpha-1}}$ and $C=\N_0 \left(1 \leq \sigma \leq 2, \frac{1}{8}, \frac{-1}{8} \notin \mathcal{R} \right)$ we see that
\begin{align*}
\N_0 \left( e^{-c\lambda^{-1}r^{-\frac{2\alpha}{\alpha-1}} V_r} |A_r \right) 
&\leq \exp \left\{-C (c-c^2\lambda^{-1}) \right\} \leq \exp \left\{-Cc/2\right\}
\end{align*}
for all $\lambda > 2c$. In particular, we can choose $c>0$ large enough so that for all such $\lambda$,
\begin{align*}
\N_0(V_r \leq c^{-1}\lambda r^{-\frac{2\alpha}{\alpha-1}}) \leq e \exp \left\{-Cc/2\right\} \leq \frac{c'}{2},
\end{align*}
where $c'$ is as in \cref{lem:Br given Ar}.
\end{proof}

In particular the results of the last two lemmas imply the following.

\begin{corollary}\label{cor:N Br Ar tie up}
There exist $c, \tilde{c} \in (0, \infty)$ such that:
\begin{enumerate}[(a)]
\item For all  $r>0, 1< \lambda < \infty$, it holds that $\N_0(B_r \text{ and } \nu (V_r) \geq \tilde{c}\lambda r^{\frac{2\alpha}{\alpha-1}}|A_r, 1< \tilde{\sigma}_r<2) \geq c$. 
\item For all $r>0, 1< \lambda < \infty$,
\[
\N_0(B_r \text{ and } \nu (V_r) \geq \tilde{c}\lambda r^{\frac{2\alpha}{\alpha-1}}, A_r, 1< \tilde{\sigma}_r<2) \geq c \lambda^{-(\alpha-1)}.
\]
\item For all $r>0, 1< \lambda < \infty$,
\[
\N_0(B_r \text{ and } \nu (V_r) \geq \tilde{c}\lambda r^{\frac{2\alpha}{\alpha-1}}, A_r, r^{\frac{2\alpha}{\alpha-1}}< \tilde{\sigma}_r<2r^{\frac{2\alpha}{\alpha-1}}) \geq c \lambda^{-(\alpha-1)} r^{-\frac{2}{\alpha-1}}.
\]
\end{enumerate}
\end{corollary}
\begin{proof}
Part (a) follows by subtracting the estimate of Lemma \ref{lem:vol set LB cond} (for the complementary event) from that of Lemma \ref{lem:Br given Ar}. Note that we can remove the restriction that $\lambda > 2\tilde{c}$ using monotonicity and adapting the constant $c$. Part (b) follows by combining (a) with Lemma \ref{lem:Ar prob LB}. Part (c) then follows by applying (b) with $r=1$, and then applying the scaling relation \eqref{eqn:Ito scaling}.
\end{proof}

Part (a) in turn leads us to the following result.

\begin{corollary}\label{cor:N Br Ar tie up normalized}
For all  $r>0, 1< \lambda < r^{-\frac{\alpha}{\alpha-1}}$, there exists $c>0$ such that $\Noo_0(B_r \text{ and } \nu (V_r) \geq \tilde{c}\lambda r^{\frac{2\alpha}{\alpha-1}}|A_r) \geq c$.
\end{corollary}
\begin{proof}
We first claim that \cref{cor:N Br Ar tie up} straightforwardly implies the following: for all $\lambda \geq 1$ it holds that $\N_0(B_r \text{ and } \nu (V_r) \geq \tilde{c}\lambda r^{\frac{2\alpha}{\alpha-1}}|A_r, 1< \sigma<3) \geq c$. For this we will see that it is sufficient to show that 
\[
\N_0(\sigma > 3|A_r, 1< \tilde{\sigma}_r<2) \leq 1 - \frac{1-c}{1-\frac{c}{2}}.
\]
By the Poisson process decomposition of \cref{prop:firstmomentrange ds} and \cite[Section VIII.1, Proposition 4]{BertoinLevy}, we can apply \cref{prop:firstmomentrange ds} and \eqref{eqn:Laplace sigma} obtain that there exists $C<\infty$ such that, for $\beta = \frac{2-\alpha}{\alpha}$,
\begin{align*}
\N_0(\sigma > 3|A_r, 1< \tilde{\sigma}_r<2) &\leq P^0 \left( (U^{(1)}+U^{(2)})_{r^2} \geq r^{\beta} | U^{(1)}_{r^2} \geq \lambda r^{\frac{2}{\alpha-1}}  \right) +  \N_0(\sigma > 3|(U^{(1)}+U^{(2)})_{r^2} < r^{\beta}))\\
&\leq C \lambda^{\alpha-1} r^{2-\beta(\alpha-1)} + (1-e^{-1})(1-\exp\{-r^{\beta}\N_0 (1-e^{- \sigma})\}) \\
&\leq C\lambda^{\alpha-1} r^{2-\beta(\alpha-1)} + (1-e^{-1})(1-\exp\{-r^{\beta}\}) 
\leq Cr^{\frac{2-\alpha}{\alpha}}.
\end{align*}
In particular we use that $\lambda <  r^{-\frac{\alpha}{\alpha-1}} \ll r^{-\frac{2}{\alpha-1}}$ to give the final line.
Hence
\begin{align*}
\N_0(\sigma > 3|A_r, 1< \tilde{\sigma}_r<2) \leq Cr^{\frac{2-\alpha}{\alpha}} \leq 1 - \frac{1-c}{1-\frac{c}{2}}
\end{align*}
for all sufficiently small $r>0$. Hence we obtain a lower bound for the complementary event and deduce that 
\[
\N_0(B_r \text{ and } \nu (V_r) \geq \tilde{c}\lambda r^{\frac{2\alpha}{\alpha-1}}|A_r, 1< \tilde{\sigma}_r<2, 1<\sigma<3) \geq 1 - (1- c)\frac{1-\frac{c}{2}}{1-c} = \frac{c}{2}.
\]
It is moreover straightforward to show that $\N_0(1<\tilde{\sigma}_r<2|A_r, 1<\sigma<3)$ is lower bounded by a constant using the same types of arguments (we omit the details) and hence that there exists $c>0$ such that, for any $r>0, 1 < \lambda < r^{-\frac{\alpha}{\alpha-1}}$,
 \begin{align*}
\N_0(B_r \text{ and } \nu (V_r) \geq \tilde{c}\lambda r^{\frac{2\alpha}{\alpha-1}}|A_r, 1<\sigma<3) &\geq C\N_0(B_r \text{ and } \nu (V_r) \geq \tilde{c}\lambda r^{\frac{2\alpha}{\alpha-1}}|A_r, 1< \tilde{\sigma}_r<2, 1<\sigma<3) \\
&\geq C.
\end{align*}
This implies the stated result by rescaling and adjusting the constants slightly.
\end{proof}

The final step is to connect $V_r$ to the ball of radius $r$ round $\pi(U)$ (recall that $V_r$ is the collection of subtrees grafted to $h$ in which the snake does not exit $[\xi_h-r/8, \xi_h+r/8]$).

\begin{lemma}\label{lem:Vr distance bound 2}
Almost surely under $\Noo_0( \cdot | A_r, B_r)$, it holds that $V_r \subset B_{\M}(\pi(U),r)$ for all sufficiently small $r$. The same is true almost everywhere under $\N_0( \cdot | A_r, B_r)$,
\end{lemma}
\begin{proof}
We start with the claim under $\Noo_0$ (it then transfers to $\N_0$ using \eqref{eqn:Ito measure integrate s}) and wlog we work on a probability space where \eqref{eqn:subsequential limits} holds almost surely. 

On the event $A_r$, let $h$ be the large hub. Consider the subtrees grafted to $h$ on the right hand side of the path from $\rho=p(0)$ to $p(U)$. We call subtrees in which the snake exits the interval $[\xi_h-\frac{r}{8}, \xi_h+\frac{r}{8}]$ ``bad"; the rest are ``good'', see Figure~\ref{fig:unif LB bad subtrees}.

\begin{figure}[h]
\begin{subfigure}{0.5\textwidth}
\includegraphics[width=7cm]{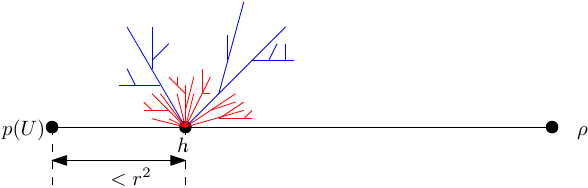}
\begin{minipage}{.1cm}
\vfill
\end{minipage}
\centering
\end{subfigure}\hfill
\begin{subfigure}{0.5\textwidth}
%\vspace{1.2em}
\includegraphics[width=7cm]{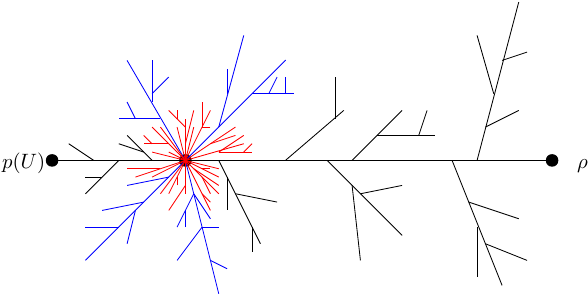}
\centering
\end{subfigure}\hfill
\caption{The setup around the hub $h$. Bad subtrees are in blue, and good subtrees are in red.}
\label{fig:unif LB bad subtrees}
\end{figure}

Hence $V_r$ consists precisely of the collection of good subtrees. Note that on the event $B_r$ this also implies that the snake does not exit $[\xi_0-3r/16, \xi_0+3r/16]$ in the good subtrees.

On the event $B_r$, the snake does not exit the interval $[Z_{\pi(U)}-\frac{r}{8}, Z_{\pi(U)}+\frac{r}{8}]$ in any of the subtrees grafted to the spine between $\pi(U)$ and $h$. It therefore follows from the fact that $D\leq\Da^{*}$ that $D(\pi(U), h) \leq \frac{r}{2}$.

In addition, each collection of consecutive good subtrees is coded by an interval $[s,t]$ for which $p(s)=p(t)=h$. It therefore similarly follows from the fact that $D\leq\Da^{*}$ that any $x$ in a good subtree satisfies $D(h,x) \leq \frac{r}{2}$. The claim therefore follows by the triangle inequality.
\end{proof}

This allows us to tie up as follows.

\begin{proof}[Proof of Proposition \ref{prop:lower vol tail bound}]
The result in (a) follows from multiplying the probabilities in \cref{lem:Ar prob LB}, \cref{cor:N Br Ar tie up normalized} and \cref{lem:Vr distance bound 2}. Similarly, one can prove using Corollary \ref{cor:N Br Ar tie up}(c) and Lemma \ref{lem:Vr distance bound 2} that there exists $c > 0$ such that for all $r>0, 1< \lambda < \infty$,
\begin{align*}
\N_0 \left( \nu (B(\pi(U),r)) \geq \lambda r^{\frac{2\alpha}{\alpha-1}} \right) &\geq c\lambda^{-(\alpha-1)}r^{-\frac{2}{\alpha-1}}.
\end{align*}
This is almost the same as the statement of part (b); to prove it as written, one in fact needs to refine the estimate of Lemma \ref{lem:vol set LB cond} slightly to consider a similar interval for the volumes there; this can be done straightforwardly by counting subtrees in the Poisson process more precisely using the same principles as in the rest of the proofs, so we have chosen to omit it.
\end{proof}

\begin{remark}\label{rmk:hub volume LB}
The same logic, combined with the result of Miermont \cite[Proposition 2]{miermont2005self} that for a hub $p(t)$
\begin{equation}
\lim_{r \downarrow 0} \frac{\Volt (B_{\Ta}(p(t), r))}{r} = \Delta_t,
\end{equation}
can be used to show that $\nu (B_{\M}(\pi(t),r))$ is typically at least of order $r^2$ when $p(t)$ is a hub of $\Ta$.
\end{remark}

\subsubsection{Proof of Corollary \ref{cor:log volume fluctuations sup LB}}

\begin{proof}[Proof of Corollary \ref{cor:log volume fluctuations sup LB}]
Again by Fubini's theorem it is sufficient to prove this for (the projection of) a uniform point $\pi(U)$. We set $r_n = 2^{-n}$ and $\lambda_n = (\log r_n^{-1})^{\frac{1-\epsilon}{\alpha-1}}$. By Proposition \ref{prop:uniform rerooting}, we can further assume that $\pi(U)$ is in fact the root $\rootm$ and we consider the spine from $\rootm$ to another independent uniform point as considered in Proposition \ref{prop:firstmomentrange ds}. Let $(\xi_s)_{s \geq 0}$ be the linear Brownian motion appearing there, and for each $n \geq 0$ set 
\[
\tau_n = \inf \{s \geq 0 : |\xi_s| \geq \epsilon r_n\}. 
\]
Note that it is indeed the case that $\tau_n<\infty$ for all sufficiently large $n$, $\N_0$-almost everywhere.

By Proposition \ref{prop:firstmomentrange ds}, the subtrees grafted to the right and left sides of the spine in the interval $[ \tau_{n+1}, \tau_n]$ are coded by a Poisson process and, conditionally on the sizes of the hubs on this part of the spine, the Poisson processes on the right and left hand sides are independent. We will consider two events, $A_n^l$ and $A_n^r$ on the left and right hand sides respectively. The event $A_n^l$ will guarantee that all vertices corresponding to time points in $[\tau_{n+1}, \tau_n]$ along the spine are within distance $r_n$ of the root. The event $A_n^r$ will ensure that there is a subtree on the right rooted at point in $[\tau_{n+1}, \tau_n]$ such that the volume of the ball of radius $r_n$ around the root in the subtree is at least $\lambda_n r_n^{\frac{2\alpha}{\alpha-1}}$. These two points combined will imply that the ball of radius $4r_n$ in the original tree has volume at least $\lambda_n r_n^{\frac{2\alpha}{\alpha-1}}$, so it is sufficient to show that $A_n^l \cap A_n^r$ occurs infinitely often, almost surely.

Let us fix $\delta \in (0,1)$. It will be helpful to consider the event
\[
B_n := \left\{ U^{(1)}_{\tau_n}-U^{(1)}_{\tau_{n+1}}, U^{(2)}_{\tau_n}-U^{(2)}_{\tau_{n+1}} \in [(\delta r_n)^{\frac{2}{\alpha-1}}, 2(\delta r_n)^{\frac{2}{\alpha-1}}] \right\},
\]
where $U^{(1)}$ and $U^{(2)}$ are as in Section \ref{sctn:spinal decomp}.

Note that the events $(B_n)_{n\geq 0}$ are independent, and moreover it follows by scaling invariance that there exists $c>0$ such that for all $n \geq 0$,
\[
P^0 \otimes \Pi_0 \left( B_n |\tau_{n} < \infty \right)>c.
\]

We then let $S^l_n$ denote the collection of subtrees that are grafted to the spine to the left of $[0,\tau_n]$, and let $S^r_n$ denote the collection of subtrees grafted to the right of $[\tau_{n+1}, \tau_n]$. We let (slightly abusing notation)
\[
A_n^l = \left\{ \sup_{x \in S_n^l} |Z_x| \leq r_n \right\}, \qquad A^r_n = \left\{ \exists T \in S^r_n : \nu (B(\rho_T, r_n) \cap T) \geq \lambda_n r_n^{\frac{2\alpha}{\alpha-1}} \right\},
\]
where here $B(\rho_T, r_n) \cap T$ denotes the part of the ball intersecting the tree $T$, measured with respect to the metric $D$ of $\M$.

We first consider a bound for the probability of $A_n^r$. On the event $B_n$ it follows from Proposition \ref{prop:lower vol tail bound}(b) that the number of such subtrees stochastically dominates a Poisson random variable with parameter $c_{\delta} \lambda_n^{-(\alpha-1)}$, and hence we deduce that there exists $c_{\delta}'>0$ such that, for all $n \geq 1$,
\[
\pr{ A_n^r |\tau_{n} < \infty,  B_n} \geq c_{\delta}' \lambda_n^{-(\alpha-1)}.
\]
Now note that, since the events $B_n \cap A_n^r$ refer to disjoint parts of the spine for distinct $n$, they form a sequence of independent events. Hence it follows from the second Borel-Cantelli lemma, the fact that $\tau_n<\infty$ for all sufficiently large $n$ $\N_0$-almost everywhere, our choice of $\lambda_n$ and the estimates above that $B_n \cap A_n^r$ occurs infinitely often, $\N_0$-almost everywhere.

We now suppose that this is indeed the case and condition on a subsequence $n_k$ such that $B_{n_k} \cap A_{n_k}^r$ occurs for all $k\geq 1$. We claim that there exists $c<\infty$ such that
\[
\N_0 (A_n^l | B_{n} \cap A_{n}^r, \tau_n < \infty) \geq 1-c\delta^{\frac{2}{\alpha-1}}.
\]
for all $n$. To see this, we apply Proposition \ref{prop:firstmomentrange ds} and \eqref{eq:rangesnake} which imply that, conditionally on $B_n$, the number of subtrees where the snake exits the interval $[-r_n,r_n]$ is stochastically dominated by a Poisson random variable with parameter $c\epsilon^{\frac{2}{\alpha-1}}$, for some $c<\infty$, and hence is non-zero with probability less than $c\epsilon^{\frac{2}{\alpha-1}}$. In particular, it therefore follows from Fatou's lemma that, conditionally on the subsequence $(n_k)_k$, 
\[
\N_0 ((A_{n_k}^l)^c \text{ for all } k \text{ large enough} | B_{n_k} \cap A_{n_k}^r \text{ i.o.}) \leq c\delta^{\frac{2}{\alpha-1}}.
\]
Moreover, as discussed above, on the event $B_n \cap A_n^l \cap A_n^r$ it holds that $\nu(B(\pi(U), 4r_n) \geq \lambda_n r_n^{\frac{2\alpha}{\alpha-1}}$. Since $B_{n_k} \cap A_{n_k}^r \text{ i.o.}$ $\N_0$-almost everywhere, we deduce that 
\[
\N_0 (\nu(B(\pi(U), 4r_{n_k}) \geq \lambda_{n_k} r_{n_k}^{\frac{2\alpha}{\alpha-1}} \text{ for all } k \text{ large enough}) \geq 1- c\delta^{\frac{2}{\alpha-1}}.
\]
Since $\delta>0$ was arbitrary, this readily implies the statement of the corollary.

\end{proof}

\section{Small volumes}\label{sctn:uniform points small}

We now turn to looking at the occurrence of unusually small volumes at typical points. As in Section \ref{sctn:UB on uniform balls}, we fix an $a > 0$ and consider the law of $M_{\epsilon}(U_a)$, where $M_{\epsilon}(s) := \epsilon^{-\frac{2\alpha}{\alpha-1}} \nu(B(s, \epsilon))$ and where $U_a$ is as in the statement of Proposition \ref{prop:upper vol tail bound lazy}.

\subsection{Upper tail bounds for small volumes}

\begin{proposition}\label{prop:lower vol tail bound lazy}
There exist $\kappa>0, C < \infty, c>0$ such that for all $r>0$, $\lambda >1$ and $a \in (r^{2}, \infty)$,
\[
\N_0 \left( \nu (B(U_a,r)) \leq \lambda^{-1} r^{\frac{2\alpha}{\alpha-1}} \right) \leq Ce^{-c\lambda^{\kappa}}.
\]
In particular this holds for $\kappa=\frac{\alpha-1}{32\alpha}$.
\end{proposition}
\begin{proof}
\textbf{Outline. }We will construct an event $E_{r, \lambda}$ on which $ \nu (B(U_a,r)) \geq \lambda^{-1} r^{\frac{2\alpha}{\alpha-1}}$ and show that the required tail bound holds for $E_{r, \lambda}$. To this end, we take $\beta, \beta', \gamma, \gamma'$ to be positive constants that will be fixed later.

The setup and the argument is similar to that of the proof of Proposition \ref{prop:lower vol tail bound}. We consider the spine from $U_a$ to $\rho$, and the associated decomposition as in Proposition \ref{prop:firstmomentdLs}. For concreteness we say that the half of the spine coming \textit{after} $U_a$ in the contour order is the right hand side of the spine. We also let $I_r$ denote the interval of the spine falling within distance $r^2 \lambda^{-2\gamma'}$ of $U_a$. We also assume that the Brownian motion appearing in Proposition \ref{prop:firstmomentdLs} is indexed by its distance from $U_a$ rather than from $\rho$. In addition we assume that the values of $Z$ are all shifted so that $Z_{U_a}=0$.

\begin{figure}[h!]
\includegraphics[width=10cm]{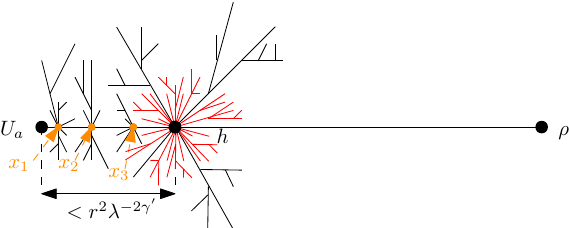}
\centering
\caption{The setup between $U_a$ and $h$ on the event $E_{r, \lambda}$. Between $U_a$ and $h$ the hubs are small and the snake has only small fluctuations. The roots of larger subtrees where the snake exits the interval $[-r \lambda^{-\gamma}, r \lambda^{-\gamma}]$ are labelled in orange.}\label{fig:Ua LB}
\end{figure}

We then let $E_{r, \lambda}$ be the intersection of the following events (see Figure~\ref{fig:Ua LB}):
\begin{enumerate}[(1)]
\item The Brownian motion $\xi$ does not exit the interval $ [-r \lambda^{-\gamma}, r \lambda^{-\gamma}]$ on the interval $I_r$.
\item There exists a hub of size at least $r^{\frac{2}{\alpha-1}} \lambda^{-\beta}$ on $I_r$. On this event we let $\thub$ be the time index of the first of these (i.e. closest to $U_a$).
\item The sum of the sizes of the hubs appearing on the interval $[0, \thub)$ is at most $r^{\frac{2}{\alpha-1}} \lambda^{\beta'-\beta}$.
\item The number of subtrees grafted to the spine in the interval $I_r$ in which the snake exits the interval $[-r \lambda^{-\gamma}, r \lambda^{-\gamma}]$ is at most $\frac{\lambda^{ \gamma}}{8}$.
\item The volume of the ball of radius $\frac{r}{2}$ centred at $h:=\pi (\thub)$ is at least $\lambda^{-1} r^{\frac{2\alpha}{\alpha-1}}$.
\end{enumerate}
\textbf{Construction. }We first explain why $ \nu (B(U_a,r)) \geq \lambda^{-1} r^{\frac{2\alpha}{\alpha-1}}$ on the event $E_{r, \lambda}$. The logic is very similar to that used to prove Proposition \ref{prop:lower vol tail bound}, and we will therefore omit some details. In particular, we let $N$ be the number of subtrees appearing in point (4) above, we order them by their (increasing) distance from the root, and for $i=1, \ldots, N$ we let $x_i$ be the root of the $i^{th}$ subtree. In addition we let $x_0 = U_a$ and $x_{N+1} = h$. Hence each $x_i$ is on the spine from $U_a$ to $\rho$, and by construction, it holds for each $i=0, \ldots, N$ that there exist $t_i, t_{i+1} \in [0, \sigma]$ such that $p(t_i)=x_i, p(t_{i+1})=x_{i+1}$, and $\sup_{r, s \in [t_i, t_{i+1}]} |Z_r-Z_s| \leq 4r\lambda^{-\gamma}$. Using that $D \leq \Da^*$ and combining with the triangle inequality therefore gives that 
\[
D(U_a, h) \leq \frac{\lambda^{\gamma}}{8} \cdot 4r\lambda^{-\gamma} = \frac{r}{2}.
\]
\textbf{Tail estimates. }We now bound the probabilities of the complements of each of the above events in turn. We will then deduce the result using a union bound.
\begin{enumerate}[(1)]
\item Using the Markov property of Brownian motion, we write
\begin{align*}
\pr{(\xi_t)_{0 \leq t \leq r^2\lambda^{-2\gamma'}} \not\subset [-r \lambda^{-\gamma}, r \lambda^{-\gamma}]} = \pr{(\xi_t)_{0 \leq t \leq 1} \not\subset [- \lambda^{\gamma'-\gamma}, \lambda^{\gamma'-\gamma}]} \leq Ce^{-c \lambda^{\gamma'-\gamma}}.
\end{align*}
\item By Proposition \ref{prop:firstmomentdLs}, the number of such hubs is lower bounded by a Poisson random variable with parameter $c \lambda^{\beta(\alpha-1)-2\gamma'}$, for an appropriate constant $c\in (0, \infty)$. Hence the probability that it is zero is at most $e^{-c \lambda^{\beta(\alpha-1)-2\gamma'}}$.
\item Let $S_{r^{\frac{2}{\alpha-1}} \lambda^{-\beta}}$ denote the sum in question. Then $S_{r^{\frac{2}{\alpha-1}} \lambda^{-\beta}}$ concerns the hubs that all occur before $h$, and which are therefore all conditioned to have size at most $r^{\frac{2}{\alpha-1}} \lambda^{-\beta}$. The number of such hubs moreover has a geometric distribution with parameter $cr^{2} \lambda^{-\beta(\alpha-1)}$, and by standard thinning properties of Poisson processes this geometric time is independent of the sizes of the hubs. In addition applying the memoryless property of the geometric distribution along with scaling invariance, it therefore follows that for any $k > 1$ (with the obvious extension of notation)
\begin{align*}
\pr{S_{r^{\frac{2}{\alpha-1}} \lambda^{-\beta}} > kr^{\frac{2}{\alpha-1}} \lambda^{-\beta}} \leq \pr{S_{r^{\frac{2}{\alpha-1}} \lambda^{-\beta}} > r^{\frac{2}{\alpha-1}} \lambda^{-\beta}}^{k/2} = \pr{S_{1} > 1}^{k/2},
\end{align*}
and so taking $k=\lambda^{\beta'}$ we deduce that the probability in question is bounded by $e^{-c\lambda^{\beta'}}$.
\item It will suffice to bound the probability of this event conditionally on the high probability event in (3), since we can exclude the alternative using the previous estimate. In this case the number of such subtrees is upper bounded by a Poisson random variable with parameter $\lambda^{\frac{2\gamma}{\alpha-1}+\beta'-\beta}$. Hence, by a Chernoff bound, the probability that it exceeds $\frac{\lambda^{\gamma}}{8}$ is upper bounded by $C\exp\{-c\lambda^{\gamma-\frac{2\gamma}{\alpha-1}-\beta'+\beta}\}$.
\item The final estimate is the most involved. We condition on the event in (2), and consider the collection of subtrees grafted to $h$ in which the snake does not exit $[\xi_h-\frac{r}{8}, \xi_h+\frac{r}{8}]$. Each of these subtrees are contained within $B(h, \frac{r}{2})$ since $D\leq \Da^*$ and since we can bound fluctuations of the snake on the interval coding the subtree. It will in fact suffice to consider only the subtrees with total volume in $[r^{\frac{2\alpha}{\alpha-1}} \lambda^{-1}, 2r^{\frac{2}{\alpha-1}} \lambda^{-1}]$. For a tree with total volume in this interval, the probability that the snake does not exit the interval $[\xi_h-\frac{r}{8}, \xi_h+\frac{r}{8}]$ is uniformly bounded below by a positive constant by scaling invariance and since $\lambda>1$. Hence the total number of trees satisfying both of these properties is Poisson with parameter $c\lambda^{1/\alpha-\beta}$, for some $c \in (0, \infty)$. Hence the probability that this number is zero is $e^{-c\lambda^{1/\alpha-\beta}}$.
\end{enumerate}
One can verify that it is possible to choose $\beta, \beta', \gamma, \gamma'>0$ so that all of the probabilities above decay as stretched exponentials in $\lambda$. One viable option (which is not optimal) is to take
\[
\beta = \frac{3}{4\alpha}, \beta'= \frac{1}{4\alpha}, \gamma=\frac{\alpha-1}{32\alpha}, \gamma'=\frac{\alpha-1}{32\alpha},
\]
which results in a final decay of $Ce^{-c\lambda^{\frac{\alpha-1}{32\alpha}}}$ via a union bound on the above events.
\end{proof}

Proposition \ref{prop:lower vol tail bound lazy} has the following consequences.

\begin{corollary}\label{cor:log volume fluctuations LB}
There exists $\kappa>0$ such that, $\N_0$-almost surely, it holds for $\nu$-almost every $x$ in $\M$ that
\[
\liminf_{r \downarrow 0} \frac{\nu (B(x,r))}{r^{\frac{2\alpha}{\alpha-1}}(\log \log r^{-1})^{-\kappa}} > 0.
\]
In particular this is true for $\kappa=\frac{32\alpha}{\alpha-1}$.
\end{corollary}
\begin{proof}
This follows by the same proof as Corollary \ref{cor:log volume fluctuations UB}. In particular, we start by fixing an arbitrary $a>0$ and proving the result for $\ell^a$-almost every $x$; this then extends straightforwardly to $\nu$-almost every $x$ using \eqref{eqn:local a def}. By Fubini's theorem, it is sufficient to prove the claim for $B(U_a,r)$, where $U_a$ is chosen according to $\ell^a$. We then take $r_n = 2^{-n}$ and $\lambda_n = c' \log \log r_n^{-1}$ where $c'>2/c$ where $c$ is the constant appearing in the statement of Proposition \ref{prop:lower vol tail bound lazy}, and then apply Borel-Cantelli, monotonicity and \eqref{eqn:local a def}.
\end{proof}

We also note the following slight extension of Proposition \ref{prop:lower vol tail bound lazy}. In what follows, $\diam (\Ta)$ refers to the diameter of $\Ta$ measured with respect to the tree metric.

\begin{proposition}\label{rmk:vol to diam}
Let $U$ be \textsf{Uniform}($[0, \sigma]$) under $\N_0$. Then there exist $\kappa, c>0, C<\infty$ such that for any $a>r^2$,
\[
\N_0 \left( \nu (B(U,r)) \leq \lambda^{-1} r^{\frac{2\alpha}{\alpha-1}} \middle| \diam (\Ta) > a \right) \leq Ce^{-c\lambda^{\kappa}}.
\]
\end{proposition}
\begin{proof}
We could have used the same argument as in Proposition \ref{prop:lower vol tail bound lazy} to obtain that there exists $\kappa, c>0, C<\infty$ such that for any $a>r^2$,
\begin{equation}\label{eqn:remark inf height}
\N_0 \left( \nu (B(\rho,r)) \leq \lambda^{-1} r^{\frac{2\alpha}{\alpha-1}} \middle| H > a \right) \leq Ce^{-c\lambda^{\kappa}}.
\end{equation}
(In particular, conditionally on $H>a$ we can consider the same decomposition to the point $U_a$ and instead construct the same event in the vicinity of $\rho$ rather than $U_a$.) Then note that, for all $a>0$,
\[
\{ H > 2a\} \subset \{\diam (\Ta) > 2a\} \subset \{ H > a\},
\]
and the tail estimate of \cite[Proposition 5.6]{goldschmidt2010behavior} shows that 
\[
\frac{ N(H > a)}{N(\diam (\Ta) > 2a)}
\]
is bounded above by a constant, uniformly in $a>0$. Hence it follows from \eqref{eqn:remark inf height} that we can modify the constant $C<\infty$ so that
\begin{equation*}
\N_0 \left( \nu (B(\rho,r)) \leq \lambda^{-1} r^{\frac{2\alpha}{\alpha-1}} \middle| \diam (\Ta) > 2a \right) \leq Ce^{-c\lambda^{\kappa}}.
\end{equation*}
for all $a > 0$. The claim then follows by uniform re-rooting invariance (Proposition \ref{prop:uniform rerooting}).
\end{proof}

\begin{corollary}\label{cor:liming global LB}
There exists $\kappa>0$ such that, $\Noo_0$-almost surely,
\begin{equation}\label{eqn:global inf in cor}
\liminf_{r \downarrow 0} \frac{\inf_{x \in M} \nu (B(x,r))}{r^{\frac{2\alpha}{\alpha-1}}(\log r^{-1})^{-\kappa}} > 0.
\end{equation}
In particular this is true for $\kappa=\frac{32\alpha}{\alpha-1}$.
\end{corollary}
\begin{proof}
We will start by proving the result under $\N_0( \cdot \cap \{\diam (\Ta)>1, K < \sigma < 2K\})$ for an arbitrary $K>0$. Choose $\epsilon>0$ and $U \sim \textsf{Uniform}([0,\sigma])$ and for each $k \leq  \lceil r^{-\frac{2\alpha+2\epsilon}{\alpha-1}} \rceil$ set $U_k = (U+k r^{\frac{2\alpha+2\epsilon}{\alpha-1}}) \mod 1$. Note that it follows from \cref{lem:height holder} that for any $\epsilon>0$, the snake tip process $(\Wt_t)_{t \in [0,1]}$ is almost everywhere $\frac{\alpha-1}{2\alpha + \epsilon}$ H\"older continuous. In particular, this implies that for all sufficiently small $r>0$, we have for all $t \in [0,1]$ that 
\begin{equation}\label{eqn:Holder snake}
\sup_{q, s \in [t, t+r^{\frac{2\alpha+2\epsilon}{\alpha-1}}]} |\Wt_s - \Wt_q| \leq r.
\end{equation}
Combined with the fact that $D \leq \Da^*$, this implies that, $\N_0$-almost surely, the collection of balls around the points $(\pi (U_k))_{0 \leq k \leq 2K\lceil r^{-\frac{2\alpha+2\epsilon}{\alpha-1}}\rceil}$ form an $\frac{r}{2}$-cover of $M$ for all sufficiently small $r>0$. Then, by a union bound and Proposition \ref{rmk:vol to diam} (if $U_k + r^{\frac{2\alpha}{\alpha-1}} \lambda^{-1}< 1 $ and $U_{k+1} + r^{\frac{2\alpha}{\alpha-1}} \lambda^{-1}> 1$ we interpret the interval below as the union of the appropriate two intervals),
\begin{align*}
&\N_0 \left( \exists k : \nu(B(U_k,r)) < r^{\frac{2\alpha}{\alpha-1}} \lambda^{-1}, \diam (\Ta) >1, K < \sigma < 2K \right) \leq 2K\lceil r^{-\frac{2\alpha+2\epsilon}{\alpha-1}} \rceil Ce^{-c\lambda^{\frac{\alpha-1}{32\alpha}}}.
\end{align*}
Hence setting $r_n = 2^{-n}$ and $\lambda_n = c'(\log r_n^{-1})^{\frac{1}{\alpha-1}}$ for some sufficiently large $c'>0$, and then applying Borel-Cantelli as in the proof of Corollary \ref{cor:log volume fluctuations UB}, we obtain that almost everywhere, $\nu(B(U_k,\frac{r_n}{2})) \geq r_n^{\frac{2\alpha}{\alpha-1}} \lambda_n^{-1}$ for all $k$ and all sufficiently large $n$. In addition, on the event that the points $(\pi (U_k))_{k}$ form an $\frac{r_n}{2}$-cover of $M$, we have that 
\[
\inf_{u \in M} \nu(B(u,r_n)) \geq \inf_k \nu(B(U_k,\frac{r_n}{2})) \geq r_n^{\frac{2\alpha}{\alpha-1}} \lambda_n^{-1}.
\]
This concludes the proof along the subsequence $r_n$ since both of the required events hold for all sufficiently large $n$, $\N_0 (\cdot \cap \{ \diam (\Ta) >1, K < \sigma < 2K\})$ almost surely. The result extends to the limit $r \downarrow 0$ using monotonicity, implying that \eqref{eqn:global inf in cor} holds under the conditioning $\N_0( \cdot \cap \{\diam (\Ta)>1, K < \sigma < 2K\})$.

The result then extends straightforwardly to $\Noo_0$ using the \Ito scaling property.
\end{proof}

\subsection{Lower tail bounds for small volumes}

We now prove some bounds in the other direction.

\begin{lemma}\label{lem:liminf local sup}
For $\nu$-almost every $x \in M$:
\[
\liminf_{r \downarrow 0} \frac{\nu (B(x,r))}{r^{\frac{2\alpha}{\alpha-1}}(\log \log r^{-1})^{-\frac{1}{\alpha-1}}} < \infty.
\]
\end{lemma}
\begin{proof}
Again by Fubini's theorem it is sufficient to prove the claim for $U$ chosen according to $\nu$. It was proved in \cite[Theorem 1.2 and Comment 1.1(c)]{duquesne2012exact} that for $\N_0$-almost every realisation of the tree $\Ta$ and $\nu$-almost every $x \in \Ta$:
\[
\liminf_{r \downarrow 0} \frac{\Volt (B_{\Ta}(U,r))}{r^{\frac{\alpha}{\alpha-1}}(\log \log r^{-1})^{-\frac{1}{\alpha-1}}} = \alpha-1.
\]
In particular, letting $A_r$ denote the event that
\[
\frac{\Volt (B_{\Ta}(U,r))}{r^{\frac{\alpha}{\alpha-1}}(\log \log r^{-1})^{-\frac{1}{\alpha-1}}} \leq 2C,
\]
we know that $A_r$ occurs infinitely often almost surely.

We use the same strategy for the volume upper bound as in the proof of Proposition \ref{prop:upper vol tail bound lazy}, and use the same notation. Take some $\delta>0$, and choose $K<\infty$ such that $\pr{{\tau}_r < Kr^2} = \pr{{\tau}_1 < K} >1- \delta$.

Now choose a subsequence $r_n \to 0$ such that $A_{Kr_n^2}$ occurs for all $n$. Since the Brownian motion on the spine is independent of the tree structure, we have that 
\[
\pr{{\tau}_{r_n} < Kr_n^2} \geq 1- \delta
\]
for all $n$. Hence by Fatou's lemma, it follows that, conditionally on the tree structure and our sequence $(r_n)_{n \geq 1}$, 
\[
\pr{{\tau}_{r_n} < Kr_n^2 \text{ i.o.}} \geq 1- \delta.
\]
Now note that, by Lemma \ref{lem:D snake upper bound}, we have (for an appropriate choice of $C_K<\infty$)
\[
\nu (B(U, r_n)) \leq \Volt (B_{\Ta}(U, Kr_n^2)) \leq C_K r_n^{\frac{2\alpha}{\alpha-1}}(\log \log r_n^{-1})^{-\frac{1}{\alpha-1}}.
\]
on the event $\{{\tau}_{r_n} < Kr_n^2\} \cap A_{Kr_n^2}$. Hence we showed that the statement of the lemma holds with probability at least $1-\delta$, and hence with probability $1$ since $\delta>0$ was arbitrary.
\end{proof}

Theorem \ref{thm:liminf fluctuations intro} now follows from Corollary \ref{cor:log volume fluctuations LB}, Corollary \ref{cor:liming global LB} and Lemma \ref{lem:liminf local sup}.

\section{Hausdorff dimension}\label{sctn:Hausdorff dim}

\subsection{Hausdorff dimension of $(\M, D)$}
In this section we identify the Hausdorff dimension for any subsequential limit of $(\Mn, n^{-\frac{2\alpha}{\alpha-1}}\dmn)$. We refer to \cite[Section 4]{mattila1999geometry} for definitions and background on Hausdorff dimension. The main result follows from combining Theorem \ref{thm:limsup fluctuations intro} with \cite[Lemma 2.1]{Duquesne2010packing} (the latter was in fact written for $\Ta$, but equally applies to $\M$).

\begin{theorem}
\sloppy $\Noo_0$-almost surely, the Hausdorff dimension of any subsequential limit of $(\Mn, n^{-\frac{2\alpha}{\alpha-1}}\dmn)_{n \geq 1}$ is equal to $\frac{2\alpha}{\alpha-1}$.
\end{theorem}

\subsection{Hausdorff dimension of $\left(\Ma , \Da^*\right)$}

Recall the definition of the metric space $\left(\Ma , \Da^*\right)$ from~\eqref{eq:def-D*} and the subsequent paragraph. Almost surely, it has the same Hausdorff dimension as any subsequential limit of $(\Mn, n^{-\frac{2\alpha}{\alpha-1}}\dmn)$:

\begin{theorem}
\sloppy Almost surely, the Hausdorff dimension of $\left(\Ma , \Da^*\right)$ is equal to $\frac{2\alpha}{\alpha-1}$.
\end{theorem}
\begin{proof}
In the whole proof, we set $d = \frac{2 \alpha}{\alpha-1}$.
\paragraph{Lower bound.}
It follows from \cite[Lemma 2.1]{Duquesne2010packing} (or more precisely, a version for $M$, but this follows by the same logic) that a sufficient condition for the Hausdorff dimension to be lower bounded by $d$ is that for every $\epsilon>0$, it holds for almost every $x$ in $\Ma$ that
\[
\limsup_{r \downarrow 0} \frac{\nu_{\alpha} (B(x,r))}{r^{d-\epsilon}} = 0.
\]
In this case, this is a straightforward consequence of Theorem \ref{thm:limsup fluctuations intro} and the fact that $D \leq \Da^*$.
\paragraph{Upper bound.}
To establish an upper bound of $d$ we need to show that for every $\epsilon>0$ we can construct a cover of $\Ma$ using balls of radius $r$ that is of size at most $r^{-(d+ \epsilon)}$ (for all sufficiently smaller $r$). For this, note it follows from \cref{lem:height holder} that for any $\epsilon>0$, the snake tip process $(\Wt_t)_{t \in [0,1]}$ is almost surely $\frac{1}{d+\epsilon}$ H\"older continuous. In particular, this implies that for all sufficiently small $r>0$, we have for all $t \in [0,1]$ that 
\begin{equation}\label{eqn:Holder snake}
\sup_{q, s \in [t, t+r^{d+2\epsilon}]} |\Wt_s - \Wt_q| \leq r.
\end{equation}
In particular, we deduce that the collection of balls around the points $(\pi (ir^{d+2\epsilon}))_{0 \leq i \leq r^{-(d+2\epsilon)}}$ form an $r$-cover of $\Ma$, as required.
\end{proof}

\addcontentsline{toc}{section}{Table of notation}
\section*{Table of notation}\label{sec:notations}

\subsection*{Finite Trees}
\vspace{-1em}
\begin{longtable}{p{3.5cm}p{12cm}}
$\mathbf{T}$ & the set of rooted plane trees \\
$\mathbf{T}_n$ & the set of rooted plane trees with $n$  edges \\
$\mu_\alpha$ & critical offspring distribution satisfying $\mu_\alpha ([x, \infty)) \sim cx^{-\alpha}$ \\
$Y$ & centred random variable on $\{-1,0,+1\}$\\
$\Tn$ & BGW tree with offspring distribution $\mu_\alpha$ conditioned to have $n$ edges \\
$u_0, \ldots, u_n$ & lexicographical ordering \\
$p$, $p_{\tf}$ & projection for a tree $\tf$ \\
$C_{\tf}, H_{\tf}, X_{\tf}$ &  contour exploration, height function, and \L ukasiewicz path of a tree $\tf$ \\
$c_{\tf}(u)$ & number of children of the vertex $u$ in the tree $\tf$ \\
%$\Tn$ & \\
$({\tf}, \ell)$ and $(\mathfrak{q}, v_*)$ & admissible labelled tree and a quadrangulation associated to it via the CVS bijection, with its distinguished vertex \\ 
$\mathbb{T}_n$ &  the set of admissible labelled trees with $n$ edges\\
%$d_{\mathfrak{q}}$ & the graph distance on $\mathfrak{q}$\\
$\Phi$ & the (reverse) pointed CVS bijection from $\mathbb{T}_n\times\{-1,+1\}$ to $\mathcal{Q}_n^{\bullet}$\\
 $(\Tn, \ell_n, \epsilon_n)$ &  BGW tree conditioned to have $n$ edges, equipped with labels with increments given by $Y$, and a uniform variable in $\{\pm 1\}$  \\
% $\dmn$  & graph distance on a map\\
\end{longtable}

\subsection*{Finite Maps}
\vspace{-1em}
\begin{longtable}{p{3cm}p{12cm}}
$\mathcal{Q}_n$, $\mathcal{Q}_n^{\bullet}$ & the set of rooted (resp. pointed and rooted) planar quadrangulations with $n$ faces\\
%$d_{\mathfrak{m}}, B_{\mathfrak{m}}(u,r)$ & graph distance on $\mathfrak{m}$ and ball of radius $r$ centred at $u$ for $d_{\mathfrak{m}}$  \\
$\left(\Mn ,  \dmn \right) $ &  random quadrangulation with $n$ faces, equipped with graph distance; image of $(\Tn, \ell_n, \epsilon_n)$ under CVS bijection\\
$\partial_n$   &  pointed vertex of $\Mn$ added by the CVS bijection\\
$\pi_n$     & canonical projection from $\{0, 1, \dots, n\}$ to $V(\Tn)=V(\Mn)\setminus\{\partial_n\}$\\
 $\rho_n$     & root vertex of $\Mn$ \\
  $\nu_n$     & uniform probability measure on $V(\Mn)$\\
%  $\Tn$ & discrete random tree with $n$ edges\\
$\dmn^{\circ}$ & pseudo-distance defined on $V(\Tn)$ and $\{0, 1,\dots, n\}$, using $\ell_n$
\end{longtable}

\subsection*{Stable Trees}
\vspace{-1em}
\begin{longtable}{p{3.5cm}p{12cm}}
$(\Ta, \dt)$ & stable tree with distance $\dt$  \\
$X$ & spectrally positive stable Lévy process with index $\alpha$\\
$\Delta_s$ & Equal to $\X_s - \X_{s^-}$, the size of the jump at time $s$ \\
$\psi$ & branching mechanism \\
$C_{\alpha}, \pi$ & opposite of drift coefficient and jump measure of $X$\\
$I_t, I_{s,t}$ & running infimum of $X$, infimum of $X$ on $[s,t]$ \\
$\X$ & normalised excursion of $X$ above its infimum at time 1\\
$N$ & \Ito excursion measure for excursions (and trees) \\
$H$ & height process\\
$(\rho_t, {\hat{\rho}}_t)_{t\geq 0}$ & exploration process and dual exploration process \\
%$\bPb_{\mu}(d \rho)$ & law of exploration process started from distribution $\mu$ \\
$m_{s,t}$ & infimum of the height function between $s\wedge t$ and $s \vee t$ \\
$\Volt$ & volume measure on $\Ta$\\
$p$ & canonical projection $[0,1] \to \Ta$\\
$[[u,v]]$ & unique geodesic from $u$ to $v$ in $\Ta$\\
%$u \prec v, u \preceq v$ & ancestral relation (the latter includes equality) \\
%$u \wedge v$ & most recent common ancestor of $u$ and $v$\\
%$[u,v]$ & projection of minimal interval $[s,t]$ such that $s \leq t, p(s)=u,p(t)=v$\\
$B_{\Ta}(x,r)$ & ball of radius $r$ around $x$ in $\Ta$ \\
$\ell^a$ & local time at level $a$ in $\Ta$ \\
\end{longtable}

\subsection*{Stable Snakes}
\vspace{-1em}
\begin{longtable}{p{3.5cm}p{12cm}}
$(Z_v)_{v \in \Ta}$ & stable snake, indexed by $\Ta$ \\
$\W$ & set of continuous real-valued functions on interval of the form $[0,\zeta]$\\
$\zeta$ & lifetime of the snake\\
$w$, $\widehat{w}=w(\zeta_w)$ &  element of $\W$, and its tip\\
$(\xi_t)_t$ & standard linear Brownian motion\\
$\Pi_x$ & law of standard linear Brownian motion started at $x$\\
$Q_{w_0}^h$ & snake on forest with height function $h$, with initial condition $w_0$\\
%$\Pb_{\mu,w}(d \rho, dW)$ & law of stable snake with initial condition $(\mu,w)$\\
%$\big((W_s(t))_{t \leq \zeta_s}\big)_{s \in [0,1]}$ & snake indexed by time \\
$(\widehat{W}_t)_{t \in [0,1]}$ & $(W_t(\zeta_t))_{t \in [0,1]}$, the tip of the snake indexed by time \\
%$\Pb_{x}$ & law of stable snake with initial condition $(0,x)$ \\
%$\Pb^*_{\mu,w}(d \rho, dW)$ & law of $(\rho,W)$ under $\Pb_{\mu,w}$, killed when $\rho$ first hits 0\\
$\N_x$ and $\Noo_x$ & excursion measures for stable snakes\\
%$\zetas, \Ws, \Wts$ & lifetime, path and tip of the snake $W$ at $s$\\
$(U^{(1)}, U^{(2)})$ & two-dimensional subordinator giving the law of the exploration and dual exploration processes along a branch to a uniform time point\\
 $(\Omega_0,\mathcal F_0, P^0)$ & probability space on which $(U^{(1)}, U^{(2)})$ is defined\\
$E^0$ & expectation associated to $(\Omega_0,\mathcal F_0, P^0)$ \\
$\tilde{\psi}$ & Laplace exponent of the marginals $U^{(1)}$ and $U^{(2)}$\\
$\psi'$ & Laplace exponent of the sum $U^{(1)}+U^{(2)}$\\
$\tilde{\pi}, \pi'$ & jump measures respectively associated to $\tilde{\psi}$ and $\psi'$\\
$(J_a, \hat{J}_a)$ & pair of random measures:  $(J_a, \hat{J}_a)(dt)=(\mathbbm{1}_{[0,a]}(t) dU_t^{(1)}, \mathbbm{1}_{[0,a]}(t) dU_t^{(2)})$\\
$\mathcal{R}$ & range of the stable snake \\
$\Pcal$ & Poisson point process with intensity $(J_h + \hat J_h)(da) \otimes \mathbb N_{\xi_a} (dw)$, describing the exploration and dual exploration of the tree and snake along the branch\\
$\mathbb{E}_{\Pcal}$ & expectation w.r.t Poisson point process $\Pcal$\\
%$\tau^D(w)$ & first exit time of domain $D$ for $w$ \\
%$\theta^D_s$ & time change on snake trajectories before their exit time from $D$ \\
%$ \widetilde{W}^D_s = W_{\theta^D_s}$ & time-changed snake\\
%$L^D_{s}$ & exit local time from $D$\\
%$\Z^D$ & exit measure from $D$\\
$s^*$ & unique time where $W$ attains its minimum \\
%$\partial$ & $\partial = p(s^*)$ \\
%$L_\sigma$ & total exit time from $(0,+\infty)$\\
%$\tilde{L}_\sigma$ & additive functional of the snake: $d \widetilde L_s = \lambda \mathbf{1}_{\{\tau(W_s) = +\infty\} } ds  + \left( \lambda^{1/\alpha} + \mu \right) dL_s$\\
%$v_{\lambda,\mu} (x), v_{\lambda} (x)$ & $ \mathbb N_x \left[ 1 - \exp(-\lambda \sigma - \mu L_\sigma)\right], \mathbb N_x \left[ 1 - \mathbbm{1}\{0 \notin \mathcal{R}\} \exp(-\lambda \sigma)\right]$ \\
%$\I, \overline{\I}$ & occupation measure of the snake $Z$ and the shifted snake $\ZZZ$ \\
%$\mathcal{R}$ & ... \\
\end{longtable}

\subsection*{Continuum Maps}
\vspace{-1em}
\begin{longtable}{p{3.5cm}p{12cm}}
$D^{\circ}_\alpha $ & pseudo-distance on $\Ta$, defined using $Z$ \\
$D^*_\alpha $ & distance on $\Ta$, satisfying $D^*_\alpha \leq D^{\circ}_\alpha $ \\
$\approx$ & equivalence relation on $\mathcal{T}_\alpha$ defined by $v \approx v' \iff D^*_\alpha(v,v')=0$ \\
$\left(\mathbf{m}_\alpha ,D^*_\alpha\right)$ & the quotient space $\T_\alpha / \approx$, equipped with distance $D^*_{\alpha}$ \\
$(\M,D)$ & subsequential scaling limit of $(V(\Mn),\dmn)$ \\
$\nu_{\alpha}$      &  volume measure on $\left(\mathbf{m}_\alpha ,D^*_\alpha\right)$ \\
$\nu$ & volume measure on $(\M,D)$ \\
%$\Vol$ & volume measure on $(\M,D)$ \purple{should be $\nu$}\\
$B(x,r)$ or $B_{\M}(x,r)$ & ball of radius $r$ around $x$ in $\M$ \\
\end{longtable}

\subsection*{Misc}
\vspace{-1em}
\begin{longtable}{p{3.5cm}p{12cm}}
$\R_+$      & $[0,+\infty)$\\
$d_{GH}(\cdot, \cdot)$, $d_{GHP}(\cdot, \cdot)$ &  Gromov-Hausdorff and Gromov-Hausdorff-Prokhorov distances\\
%$(R_t)_{t \geq 0}$, $P_x^{(d)}$, $\mathbb{E}_x^{(d)}$ &  Bessel process of dimension ${2+2\tilde{\nu}}$, its law and expectation \\
%$\tilde{\lambda}, \tilde{\nu}$ & $\tilde{\lambda}^2 = \frac{2 \alpha (\alpha +1)}{(\alpha -1)^2}$, $\tilde{\nu} = \frac{3 \alpha +1}{2(\alpha -1)}$ \\
% ... & ... \\
% $\mathcal{T}$ & continuum tree\\
% $N$ & \Ito excursion measure for excursions (and trees) \\
% $\No$ & normalized \Ito excursion measure \\
% $\Noo$ & law of exploration process and snake under $\No$ \\
% $\sigma$ & lifetime of excursion \\
% $H$ & height process \\
$M_f(\R_+)$ & set of finite measures on $\R_+$\\

\end{longtable}

\newpage

\addcontentsline{toc}{section}{References}
\printbibliography

\end{document}